\documentclass{aims}
\usepackage{amsmath}
  \usepackage{paralist}
  \usepackage{graphics} 
  \usepackage{epsfig} 
\usepackage{graphicx}  \usepackage{epstopdf}
 \usepackage[colorlinks=true]{hyperref}
 \usepackage{appendix}
\hypersetup{urlcolor=blue, citecolor=red}

\def\currentvolume{X}
 \def\currentissue{X}
  \def\currentyear{200X}
   \def\currentmonth{XX}
    \def\ppages{X--XX}
     \def\DOI{10.3934/xx.xx.xx.xx}

\newtheorem{theorem}{Theorem}[section]
\newtheorem{corollary}{Corollary}

\newtheorem{lemma}[theorem]{Lemma}
\newtheorem{proposition}[theorem]{Proposition}

\theoremstyle{definition}

\newtheorem{remark}[theorem]{Remark}

\title[Asymptotic spreading of interacting species] 
      {Asymptotic spreading of interacting species with multiple fronts I: A geometric optics approach}

\author[Qian Liu, Shuang Liu and King-Yeung Lam]{}


 \email{liuqian0519@ruc.edu.cn}
 \email{liushuangnqkg@ruc.edu.cn}
 \email{lam.184@osu.edu}

\thanks{The last author is partiallly supported by  NSF grant DMS-1853561.}

\thanks{$^*$ Corresponding author: King-Yeung Lam}

\begin{document}
\maketitle
\centerline{\scshape Qian Liu$^{1,2}$, Shuang Liu$^{1,2}$ and King-Yeung Lam$^{2,*}$}
\medskip
{\footnotesize
 \centerline{ $^1$ Institute for Mathematical Sciences, Renmin University of China}
   \centerline{Beijing, 100872, China}
} 

\medskip

\medskip
{\footnotesize
   \centerline{$^2$ Department of Mathematics, Ohio State University}
   \centerline{Columbus, OH 43210, USA}
}

%
%

\bigskip

 \centerline{(Communicated by the associate editor name)}

\begin{abstract}
We establish spreading properties of the Lotka-Volterra competition-diffusion system. When the initial data vanish on a right half-line, we derive the exact spreading
speeds and prove the convergence to homogeneous equilibrium states between successive invasion fronts. Our method is inspired by the geometric optics approach for Fisher-KPP equation due to Freidlin, Evans and Souganidis. Our main result settles an open question raised by Shigesada et al. in 1997,  and shows that one of the species spreads to the right with a nonlocally pulled front.
\end{abstract}

\section{Introduction}

In this paper, we study the spreading of two competing species, modeled by
 the Lotka-Volterra two-species competition-diffusion system. The non-dimensionalized system reads
\begin{equation}\label{eq:1-1}
\left\{
\begin{array}{cc}
\partial_t u-\partial_{xx}u=u(1-u-av),& \text{ in }(0,\infty)\times \mathbb{R},\\
\partial _t v-d\partial_{xx}v=r v(1-bu-v),& \text{ in }(0,\infty)\times \mathbb{R},\\
u(0,x)=u_0(x), & \text {for all } x \in \mathbb{R},\\
v(0,x)=v_0(x), & \text {for all } x \in \mathbb{R},
\end{array}
\right .
\end{equation}
where the positive constants $d$ and $r$ are the diffusion coefficient and intrinsic
growth rate of $v$; $u(t,x)$ and $v(t,x)$ represent the population densities of the two competing species at time $t$ and location $x$. Without loss of generality, we assume {$dr \geq 1$} throughout most of this paper.
It
is clear that \eqref{eq:1-1} admits  a trivial equilibrium $(0, 0)$ and two semi-trivial equilibria (1,0) and (0,1). Throughout this paper we assume  $0 <a<1$ and $0<b <1$, so that there is a further linearly stable equilibrium:
 $$(k_1,k_2)=\left(\frac{1-a}{1-ab},\frac{1-b}{1-ab}\right).$$

 There is a vast number of mathematical results concerning the spreading of competing populations
with a single interface connecting two equilibrium states, see, e.g., \cite{Lewis_2002,Liang_2007,Lin_2012} and the references therein. By a classical result by Lewis et al., it is known that for \eqref{eq:1-1},
the  spreading speed is closely related to the minimum wave speed of traveling wave solutions connecting the ordered pair of two equilibria of \eqref{eq:1-1}.
\begin{theorem}[Lewis et al.{\cite{Lewis_2002,Li_2005}}]\label{thm:LLW}
Let $(u,v)$ 
be the solution of \eqref{eq:1-1} with initial data
$$u(0,x)=\rho_1(x), \,\,\,\, v(0,x)=1-\rho_2(x), 
$$
where $0\leq \rho_i<1$ $(i=1,2)$ are compactly supported functions in $\mathbb{R}$. Then there exists ${c}_{\textup{LLW}} \in [2\sqrt{1-a}, 2]$ such that 
 \begin{align*}
 \left\{
\begin{array}{l}
\smallskip
\lim\limits_{t\rightarrow \infty}  \sup\limits_{|x|<c t} (|u(t,x)-k_1|+|v(t,x)-k_2|)=0 ~\text{ for each } c<{c}_{\textup{LLW}},\\
\lim\limits_{t\rightarrow \infty}  \sup\limits_{|x|>c t} (|u(t,x)|+|v(t,x)-1|)=0 ~\text{ for each }~ c>{c}_{\textup{LLW}}.
\end{array}
\right.
 \end{align*}
In this case, we say that the population $u$ spreads at speed ${c}_{\textup{LLW}}$.
\end{theorem}
\begin{remark}\label{rmk:LLW}
If the initial data $(u,v)(0,x)$ is a compact perturbation of $(1,0)$, then there exists $\tilde{c}_{\textup{LLW}} \in [2\sqrt{dr(1-b)}, 2\sqrt{dr}]$ such that the species $v$ spreads at speed $\tilde{c}_{\textup{LLW}}$.
\end{remark}
Concerning the bounds of ${c}_{\textup{LLW}}$, standard linearization near the equilibrium $(0, 1)$ shows
 $${c}_{\textup{LLW}}\geq 2\sqrt{1-a}.$$
 Numerical tests by Hosono [2]  showed that the above equality holds only for certain values of model parameters $d,r,a,b$. This begs the question of if and when the equality holds, which is known as the question of linear determinacy.
 Recently, Huang and Han \cite{Huang_2011} rigorously demonstrated that $c_{\rm LLW} > 2\sqrt{1-a}$ is possible via an explicit construction. On the other hand,  sufficient conditions for linear determinacy are first obtained in \cite{Lewis_2002} and are subsequently improved in \cite{Huang_2010}. See also \cite{Ou2019a,Ou2019} for recent development on necessary and sufficient conditions.

The goal of this paper is to understand the co-invasion of two competing species for a different class of initial data $(u_0,v_0) \in C(\mathbb{R}; [0,1])^2$:
$$
{\rm{(H_{\infty})}}\begin{cases}
\text{There exist positive constants }\theta_0, x_0\text{ such that } \\
\theta_0 \leq u_0(x) \leq 1 \quad \text{ in }(-\infty,0], \quad \text{ and }\quad u_0(x) =0 \quad \text{ in  }[x_0,\infty).\\
\text{Also, }v_0(x)\text{ is non-trivial and has compact support.}
\end{cases}
$$
In other words, we assume the right habitat is unoccupied initially. This question was raised by Shigesada and Kawasaki \cite{Shigesada_1997} as they considered the invasion of two or more tree species into the North American continent at the end of last ice age (approximately 16,000 years ago) \cite{Davis_1981}. An interesting scenario arises when the slower moving species invades into the (still expanding) range of the faster moving species. The numerical computations in \cite[Ch. 7]{Shigesada_1997} illustrate that the two species set up at least two invasion fronts: The first front occurs as the faster species invades into open habitat at some speed $c_1$, while the next front appears when the slower species ``chases'' the faster species at speed $c_2$.

When the initial data $u_0$ and $v_0$ are both compactly supported,
 the spreading properties of \eqref{eq:1-1} with $a,\,b\in(0,1)$ were initially studied by  Lin and Li \cite{Lin_2012}.
 They showed that the faster species $v$ spreads at speed $c_1 = 2\sqrt{dr}$ and  obtained an estimate of the spreading speed $c_2$ of the slower species $u$, which satisfies $2\sqrt{1-a} \leq c_2 \leq 2$. In case $2\sqrt{dr} > 2+2\sqrt{1-a(1-b)}$, they obtained an improved estimate of $c_2$, namely,
 $$
 2\sqrt{1-a} \leq c_2 \leq 2\sqrt{1-a(1-b)}.
 $$
 Nevertheless, the exact formula of $c_2$ remained open.

In the bistable case $a,b > 1$ {with appropriate initial conditions}, the spreading problem was studied by Carr\`{e}re \cite{Carrere_2018}, who showed that solutions of \eqref{eq:1-1} exhibit two moving interfaces connecting, starting from the right, $(0,0)$ to $(0,1)$ to $(1,0)$. The first interface moves with the expected speed of $c_1=2\sqrt{dr}$, whereas the second interface moves with speed $c_2$ which is the speed of the unique traveling wave solution connecting $(0,1)$ to $(1,0)$.

The monostable case $0 < a < 1 < b$, which is closely related to our problem, was
considered by Girardin and the last author \cite{Girardin_2018}. By a delicate construction of piecewise smooth super- and sub-solutions for \eqref{eq:1-1}, it was shown that while the faster species spreads at the expected speed $c_1=2\sqrt{dr}$, the spreading speed $c_2$ of the slower species depends on $c_1$ in a non-trivial way, and is a {nonlocally} determined quantity in general
(see Subsection \ref{subsect:1.1} for details). While it is possible to generalize the method in \cite{Girardin_2018} for our purpose, the details of the construction will likely be quite daunting, as a total of three moving interfaces, connecting $(0,0)$, $(0,1)$, $(k_1,k_2)$ and $(1,0)$, has to be accounted for. Hence, a more direct method is preferable to better understand the problem.

In this paper, we will demonstrate how the geometric optics point of view can
 lead to a more direct determination of the various spreading speeds of the competing species.
The method of geometric optics is based on deriving the limiting problem for large space and large time, for which the solution has to be understood in the viscosity sense. It was introduced by Freidlin \cite{Freidlin_1985}, who employed probabilistic arguments to study the asymptotic behavior of solution to the Fisher-KPP equation modeling the population of a single species. Subsequently, the result was generalized by Evans and Souganidis using PDE arguments; see also \cite{Barles_1993,Berestycki_2012,Majda_1994,Mendez_2003,Souganidis_1995,Xin_2000}. The method was also applied by Barles, Evans and Souganidis \cite{Barles_1990} to study {{\rm{KPP}} systems, where several species spread at a common spreading speed.}

Finally, we also mention some related works on the Cauchy problem of interacting species spreading into open habitat. A class of predator-prey systems were considered by Ducrot et al. \cite{Ducrot_preprint}. For cooperative systems with equal diffusion coefficients, the existence of stacked fronts for cooperative systems was also studied by Iida et al. \cite{Iida_2011}. We refer to \cite{Li_2018} for the spreading of two species into an open habitat in an integro-difference competition model. Therein results analogous to Theorem \ref{thm:1-1} were established in the case  $c_2 = c_{\rm{LLW}} = 2\sqrt{1-a}$ i.e., in case  $c_1 \gg c_2$ and that linear determinacy holds.  In these works, however, the spreading speeds of individual species can be determined locally and are not influenced by the presence of other invasion fronts.

\subsection{Main results}\label{subsect:1.1}
Our main result, for the case $dr >1$, can be stated as follows.

\begin{theorem}\label{thm:1-1} 
Assume $dr >1$. Let $(u,v)$ be any solution of \eqref{eq:1-1} such that the initial data satisfies
$\rm{(H_{\infty})}$.  Then there {exist} $c_1, c_2, c_3\in \mathbb{R}$ such that
\begin{itemize}

\item[{\rm(a)}] $c_3 \leq -2\sqrt{dr(1-b)} < 0 < 2\sqrt{1-a} \leq c_2 \leq 2 < c_1;$

\item[{\rm(b)}] For each small $\eta>0$, the following spreading results hold:
\begin{equation}
\begin{cases}
\lim\limits_{t\rightarrow \infty} \sup\limits_{ x>(c_{1}+\eta) t} (|u(t,x)|+|v(t,x)|)=0, \\
\lim\limits_{t\rightarrow \infty} \sup\limits_{(c_2+\eta) t< x<(c_{1}-\eta) t} (|u(t,x)|+|v(t,x)-1|)=0, \\\lim\limits_{t\rightarrow \infty} \sup\limits_{(c_{3}+\eta)t< x<(c_2-\eta) t} (|u(t,x)-k_1|+|v(t,x)-k_2|)=0 , \\ \lim\limits_{t\rightarrow \infty} \sup\limits_{x<(c_{3}-\eta)t} (|u(t,x)-1|+|v(t,x)|)=0. 
 \end{cases} \label{eq:spreadingly}
\end{equation}
\end{itemize}
Precisely, the spreading speeds $c_1,\,c_2,\,c_3$ can be determined as follows:
\begin{equation}\label{eq:c1c2c3}
c_1 = 2\sqrt{dr}, \quad c_2 = \max\{{c}_{\textup{LLW}}, {c}_{\textup{nlp}}\},\quad c_3 = -\tilde{c}_{\textup{LLW}},
\end{equation}
where ${c}_{\textup{LLW}}$ $(resp.~\tilde{c}_{\textup{LLW}})$ is the spreading speed of $(k_1,k_2)$ into $(0,1)$ $($resp. $(k_1,k_2)$ into $(1,0))$ as given in Theorem \ref{thm:LLW} $($resp. Remark \ref{rmk:LLW}$)$.  And
\begin{equation}\label{eq:c_acc-1}
\begin{array}{ll}
{c}_{\textup{nlp}}= \left\{
\begin{array}{ll}
\sqrt{dr} - \sqrt a + \frac{1-a}{\sqrt{dr} - \sqrt a},  & \qquad  \qquad \text{ if }\,\,\sqrt{dr} \leq \sqrt a +\sqrt{1-a},\\
2\sqrt{1-a}, & \qquad \qquad \text{ if }\,\,\sqrt{dr} > \sqrt a +\sqrt{1-a}.
\end{array}
\right.
\end{array}
\end{equation}
\end{theorem}
The above result can be abbreviated as
\begin{equation}\label{eq:spreadinglyly}
(1,0) \,\, \xleftarrow{\,\,\, c_3 \,\,\,}\,\,(k_1,k_2)\,\, \xrightarrow{\,\,\, c_2 \,\,\,}\,\,(0,1)\,\, \xrightarrow{\,\,\, c_1 \,\,\,}\,\,(0,0).
\end{equation}

The above result also shows that, while the spreading speed $c_1$ of the faster species $v$ is the linearly determined speed of $2\sqrt{dr}$ and is unaffected by the slower species $u$, the corresponding speed $c_2$ of species $u$ is a non-{increasing} function of $c_1$. This is due to the fact that the presence of $v$ negatively impacts the invasion of $u$. It is clear that, even though $u_0(x)$ vanishes for $x \gg 1$, the spreading speed $c_2$ can be strictly greater than ${c}_{\textup{LLW}}$, i.e., the second front {moves} at an enhanced speed that is strictly greater than the minimal speed of traveling wave solutions.
As we shall see, the expression \eqref{eq:c_acc-1} of ${c}_{\textup{nlp}}$ coincides with that in \cite[Theorem 1.1]{Girardin_2018}, and  can be characterized by
\begin{equation}\label{eq:equivdef}
\{(t,x):\, w_1 (t,x) = 0\} = \{ (t,x):\, t>0\, \text{ and }\, x \leq c_{\textup{nlp}} t\},
\end{equation}
where $w_1$ is the unique viscosity solution of the following Hamilton-Jacobi equation with space-time inhomogeneous coefficients:
\begin{equation}\label{eq:hj_j11}
\left\{\begin{array}{ll}
\medskip
\min\{\partial_t  w + |\partial_x w|^2 + 1 - a \chi_{\{x \leq 2\sqrt{dr} t\}}, w\} = 0, & \text{ for }(t,x) \in (0,\infty)\times \mathbb{R},\\
 w(0,x)=\left\{\begin{array}{ll} 0, &\text{ for } x\in (-\infty,0],\\
\infty, &\text{ for } x\in  (0,\infty),
\end{array}
\right.
\end{array}
\right.
\end{equation}
where $\chi_{S}$ is the indicator function of the set $S \subset \mathbb{R}^2$. (Hereafter the initial condition similar to the one in \eqref{eq:hj_j11} is to be understood in the sense that $w_1(t,x) \to 0$ if $(t,x) \to (0, x_0)$ for some $x_0 <0$, and $w_1(t,x) \to \infty$ if $(t,x) \to (0,x_0)$ for some $x_0 >0$).

To explain the sense in which the speed ${c}_{\textup{nlp}}$ is said to be nonlocally determined, let us define $c_{\textup{nlp}}$ for the moment by the relation \eqref{eq:equivdef}, where $w_1$ is the unique viscosity solution of \eqref{eq:hj_j11}.
As we will show in Lemma \ref{lem:3-3}, $w_1(t,x)=\max\{J_1(t,x),0\}$, with
$$
J_1(t,x) = \inf_{\gamma(\cdot)}\int_0^t \left[\frac{|\dot\gamma(s)|^2}{4} - 1 + a\chi_{\{\gamma(s) \leq 2\sqrt{dr} s\}} \right]\,ds,
$$
where the infimum is taken over all curves $\gamma \in H^1_{\rm{loc}}([0,\infty))$ such that $\gamma(0) = 0$ and $\gamma(t) = x$. For each $(t,x)$ on the front, (i.e., $x = c_{\rm nlp}t$), the minimizing path $\hat\gamma(s) = \hat\gamma^{t,x}(s)$ describes how an individual located at $(0,0)$ 
arrives at the front $x = c_{\rm nlp}t$ at time $t$.

Now, when $a=0$, the problem \eqref{eq:hj_j11} is homogeneous.
In this case $w_1(t,x) = \frac{t}{4}\left(\frac{x^2}{t^2} -4\right)$,
so that the front is characterized by $x = 2t$. Furthermore, for each $(t,x)$ on the front, the minimizing path $\hat\gamma(s)$ is given by the straight line $\hat\gamma(s) = \frac{x}{t}s = 2s$, i.e., an individual arriving at the front $x = 2t$ at time $t$ has been staying at the front $x = 2s$  for any previous time $s \in [0,t]$. Hence, we say that the spreading speed is locally determined in case $a=0$.

Consider instead the problem \eqref{eq:hj_j11} in case $a \in (0,1)$. Then the minimizing paths are not straight lines in general. In fact, for $1<\sqrt{dr} < \sqrt{a} + \sqrt{1-a}$, if an individual finds itself at the moving front at time $t$, i.e., $x = c_{\rm nlp} t$, then the corresponding minimizing path is a piecewise linear curve connecting $(0,0)$, $(\tau, 2\sqrt{dr} \tau)$, and $(t,x)$, for some $\tau \in (0,t)$ (see Appendix \ref{sec:A} for details). Hence, the individual arriving at the front $x = c_{\rm nlp}t$ at time $t$ does not stay on the front in previous time. In fact, it spends a significant amount of time ahead of the front (by moving with speed $2\sqrt{dr}>c_{\rm nlp}$). Thus the speed $c_{\rm nlp}$ is affected by the quality of habitat well ahead of the actual front, and we say that it is nonlocally determined. In fact, it is nonlocally pulled (see, e.g., \cite{Roques_2015} for the meaning of pulled versus pushed fronts).

We also mention a closely related work, due to Holzer and Scheel \cite{Holzer_2014}, which includes among others the special case $b=0$ of \eqref{eq:1-1}. Their proof relies on linearization at a single moving frame $y=x-2t$ where the linearized problem becomes temporally constant. Such a problem was also studied by \cite{Berestycki_2018,Fang_2016}, where the complete existence and multiplicity of forced traveling waves as well as their attractivity, except for some critical cases, were obtained.  In contrast, our approach can be applied to problems with coefficients depending on multiple moving frames $x - c_i t$ for several $c_i$. This allows the treatment of the spreading of three competing species with different speeds, which will appear in our forthcoming work.

Using Theorem \ref{thm:1-1}, which treats the case $dr>1$, we can derive the following results concerning the remaining cases $dr =1$ and $0 < dr < 1$.
\begin{theorem}\label{thm:1-1'} 
Assume $dr =1$. Let $(u,v)$ be the solution of \eqref{eq:1-1} such that the initial data satisfies
$\rm{(H_{\infty})}$.  Then for each small $\eta>0$, 
\begin{equation}\label{eq:spreadingly1}
\begin{cases}
\lim\limits_{t\rightarrow \infty} \sup\limits_{ x>(2+\eta) t} (|u(t,x)|+|v(t,x)|)=0, \\
\lim\limits_{t\rightarrow \infty} \sup\limits_{(-\tilde{c}_{\textup{LLW}}+\eta)t< x<(2-\eta) t} (|u(t,x)-k_1|+|v(t,x)-k_2|)=0 , \\
 \lim\limits_{t\rightarrow \infty} \sup\limits_{x<(-\tilde{c}_{\textup{LLW}}-\eta)t} (|u(t,x)-1|+|v(t,x)|)=0,  \end{cases}
\end{equation}
where $\tilde{c}_{\textup{LLW}}$ defines the spreading speed of $(k_1,k_2)$ into $(1,0)$ for the system \eqref{eq:1-1} with $dr =1$ as given in Remark \ref{rmk:LLW}.
\end{theorem}
The above result can be abbreviated as
\begin{equation*}
(1,0) \,\, \xleftarrow{\,\,\, -\tilde{c}_{\textup{LLW}} \,\,\,}\,\,(k_1,k_2)\,\, \xrightarrow{\,\,\, 2 \,\,\,}\,\,(0,0).\,\
\end{equation*}
Theorem \rm{1.4} implies the invasion process from  $(k_1,k_2)$ into $(0,0)$  does exist, which is related to the results in Tang and Fife \cite{Tang_1980}  where the existence of traveling wave solutions of \rm{(1.1)} connecting $(k_1,k_2)$ to $(0,0)$  was proved.

By switching the roles of $u$ and $v$, it is not difficult to derive the following result in case $dr <1$.
\begin{corollary}\label{rmk:drleq1}
In case $dr <1$, the transition of equilibria becomes $$
(1,0) \,\, \xleftarrow{\,\,\, c_3 \,\,\,}\,\,(k_1,k_2)\,\, \xrightarrow{\,\,\, c_2 \,\,\,}\,\,(1,0)\,\, \xrightarrow{\,\,\, c_1 \,\,\,}\,\,(0,0).
$$
Precisely, the spreading speeds $c_1,\,c_2,\,c_3$ can be determined as follows:
$$c_1 = 2, \,c_2 = \max\{\tilde{c}_{\textup{LLW}},\tilde {c}_{\textup{nlp}}\} \, \text{ and }  c_3 = -\tilde{c}_{\textup{LLW}}, $$
where
 \begin{equation}\label{eq:tildec_acc-1}
\tilde {c}_{\textup{nlp}}= \left\{
\begin{array}{ll}
\medskip
1-\sqrt{drb}+ \frac{dr(1-b)}{1-\sqrt{drb}},  & \qquad  \qquad \text{ if }\,\,1 \leq \sqrt {dr}(\sqrt{b}+\sqrt{1-b}),\\
2\sqrt{dr(1-b)}, & \qquad \qquad \text{ if }\,\,1 >\sqrt {dr}(\sqrt{b}+\sqrt{1-b}).
\end{array}
\right.
\end{equation}
\end{corollary}

\begin{remark}
As in \cite{Evans_1989}, our approach can be applied to the spreading problem of competing species in higher dimensions under minor modifications. However, we choose to focus here on the one-dimensional case to keep our exposition simple, and close to the original formulation of the conjecture in \cite[{Ch. 7}]{Shigesada_1997}.
\end{remark}

\begin{remark}
 We also mention here some related works concerning competition systems \cite{Du_2018,Guo_2015,Liu_2019,Wang_2017,Wang_2018,Wu_2015} with Stefan-type moving boundary conditions. Therein some estimates of asymptotic speeds of the moving boundaries were proved. In contrast to the Cauchy problem considered here, there are no far-fields effect in such moving boundary problems.
\end{remark}

\subsection{Numerical simulation of main results}
The asymptotic behaviors of the solutions to \eqref{eq:1-1} for the three cases: \rm{(a)} $dr>1$, \rm{(b)} $dr=1$, \rm{(c)} $0<dr<1$ are illustrated in Figure \ref{figure1}. Precisely, {\rm{(a)}} with $d=1.5$ shows that the solutions of \eqref{eq:1-1}
behave as predicted by Theorem \ref{thm:1-1}. Therein, species $v$ spreads faster than species $u$, i.e., $c_1=2\sqrt{dr}>2\geq c_2$. \rm{(b)} with $d=1$ corresponds to Theorem \ref{thm:1-1'}, where $c_1=c_2=2$. Finally, \rm{(c)} with $d=0.5$ means that species $u$ spreads faster than species $v$, i.e., $c_1=2>c_2$ as discussed in Corollary \ref{rmk:drleq1}. Due to the limitation of our methods, we can't get the asymptotic profiles of \eqref{eq:1-1}.
\begin{figure}[htp]
\begin{center}
\includegraphics[width=1.0\textwidth]{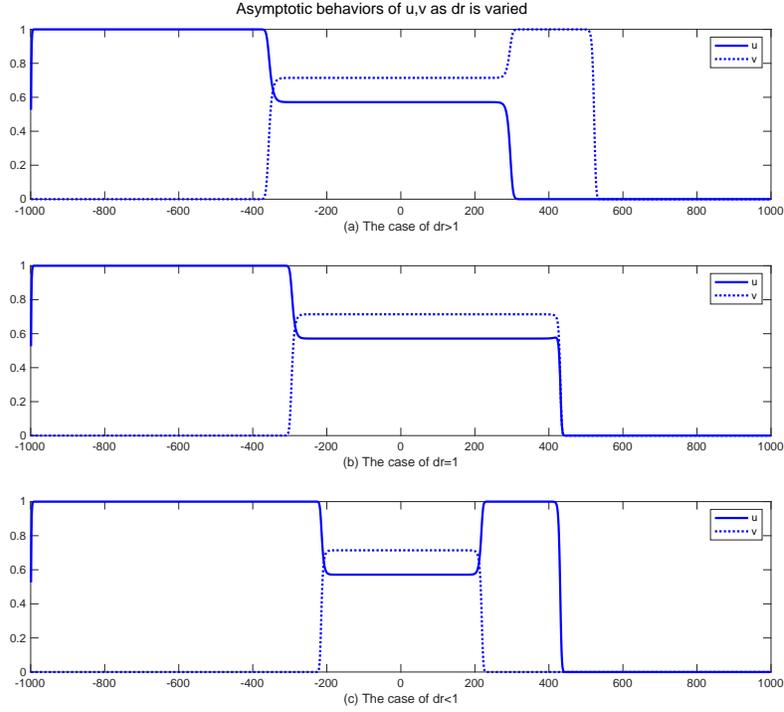}
 \caption{Asymptotic behaviors of the solutions to \eqref{eq:1-1} with $a=0.6,\, b=0.5,\,r=1$, and $d=1.5$ in \rm{(a)}, $d=1$ in \rm{(b)}, $d=0.5$ in \rm{(c)}, where the initial data are chosen as  $u(0,x)=\chi_{[-1000,0]}$ and $v(0,x)=\chi_{[-20,0]}$.
} \label{figure1}
 \end{center}
\end{figure}

In what follows, we present some numerics to illustrate the formulas of $c_1,c_2$ and $c_3$ given in Theorem \ref{thm:1-1}. 
Set $a=0.6,\, b=0.5,\,r=1$ and $d=1.5$ as in Figure \ref{figure1}{\rm{(a)}}, whereby the sufficient conditions for linear determinacy given by \cite[Theorem 2.1]{Lewis_2002} are satisfied.  The theoretical results in Theorem \ref{thm:1-1} assert that
\begin{equation*}
  \begin{cases}
c_1=2\sqrt{dr}\approx 2.4495,\\ 
c_2=\max\{c_{\textup{LLW}},c_{\text{nlp}}\} \approx\max\{1.265,1.3387\}
\approx1.3387,\\
c_3=\tilde{c}_{\textup{LLW}}=-2\sqrt{dr(1-b)}\approx-1.7321,
  \end{cases}
\end{equation*}
where, in determining $c_2$, we used the facts that {\rm(i)} $c_{\textup{LLW}}=2\sqrt{1-a}\approx 1.265$ is linearly determined \cite[Theorem 2.1]{Lewis_2002}; {\rm(ii)} $\sqrt{dr}\approx 
1.2248<\sqrt{a}+\sqrt{1-a}\approx1.407$ so that $c_{\text{nlp}}=\frac{c_1}{2}-\sqrt{a}+\frac{1-a}{\frac{c_1}{2}-\sqrt{a}} \approx 1.3387$.

 Denote
 $$x_1(t)=\sup\{x~|~v(t,x)>0.6\},\quad x_2(t)=\sup\{x~|~u(t,x)>0.4\},$$
 $$\quad  x_3(t)=\inf\{x~|~u(t,x)<0.7\}.$$

The graphs of $x_i(t)/t$ ($i=1,2,3 $) are shown in Figure \ref{figure2}. They indicate that, indeed, $x_i(t)/t\to c_i$ as $t\to\infty$. In fact, at $t=200$,  $x_1(t)/t\approx 2.4452$ comparing to the theoretical value $c_1\approx 2.4495$; $x_2(t)/t \approx 1.3695$  comparing to $c_2\approx1.3387$; and $x_3(t)/t \approx -1.7214$ comparing to $c_3\approx-1.7321$ in Theorem \ref{thm:1-1}. Note that we expect an error of $O(t^{-1}\log t)$ between the approximated value $x_i(t)/t$ and the theoretical value $c_i$.  Thus the formulas of $c_1,c_2,c_3$ provided in Theorem \ref{thm:1-1} are confirmed by Figure \ref{figure2}.

\begin{figure}[htp]
\begin{center}
  \includegraphics[width=1.0\textwidth]{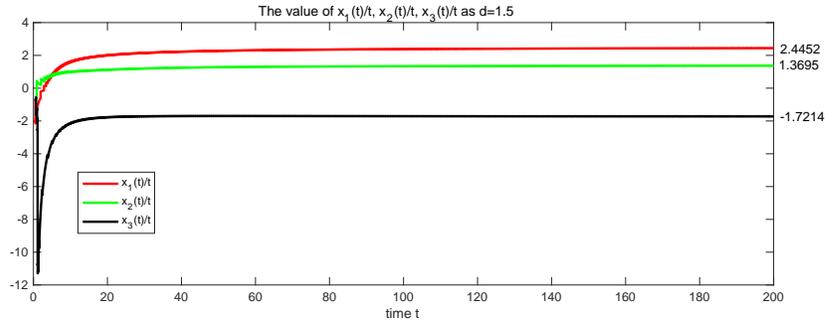}\\
   \caption{The graphs of $x_i(t)/t$ ($i=1,2,3 $) with $a=0.6,\, b=0.5,\,r=1$ and $d=1.5$ where the initial data are chosen as  $u(0,x)=\chi_{[-1000,0]}$ and $v(0,x)=\chi_{[-20,0]}$. } \label{figure2}
  \end{center}
\end{figure}

\subsection{Outline of main ideas}
We outline the main steps leading to the determination of the nonlocally pulled spreading speed $c_2$, as stated Theorem \ref{thm:1-1}. (The other spreading speeds $c_1, c_3$ can be determined by standard methods as in \cite{Girardin_2018}, see Proposition \ref{prop:1}.)
\begin{enumerate}
\item  To estimate $c_2$ from below, we consider the transformation 
$$w^\epsilon(t,x) = -\epsilon \log u\left( \frac{t}{\epsilon}, \frac{x}{\epsilon}\right)$$ and show that the half-relaxed limits
$$
w_*(t,x) =  \hspace{-0.3cm} \liminf_{\scriptsize \begin{array}{c}\epsilon \to 0\\ (t',x') \to (t,x)\end{array}} \hspace{-0.3cm} w^\epsilon(t',x') \quad \text{ and } \quad w^*(t,x) =  \hspace{-0.3cm}  \limsup_{\scriptsize \begin{array}{c} \epsilon \to 0 \\ (t',x') \to (t,x) \end{array}}  \hspace{-0.3cm}  w^\epsilon(t',x')$$ exist, upon establishing uniform bounds in $C_{\mathrm{loc}}((0,\infty) \times \mathbb{R})$ (see Lemma \ref{lem:3-1}). By  the comparison principle, we show that
\begin{equation}\label{eq:upperlimitw}
0 \leq w_* \leq w^* \leq  w_1 \quad \text{ in }(0,\infty)\times\mathbb{R},
\end{equation}
where $w_1$ is the viscosity solution of the Hamilton-Jacobi equation \eqref{eq:hj_j11}. Solving
$w_1$ explicitly by way of its
 variational characterization, we have
$$
\{(t,x):w_1(t,x) =0\} = \{ (t,x): t>0 \, \text{ and }\, x \leq  {c}_{\textup{nlp}} t\}.
$$
Thus $w^\epsilon(t,x) = -\epsilon \log u\left(\frac{t}{\epsilon}, \frac{x}{\epsilon}\right) \to 0$ in $\{(t,x): \, x <  {c}_{\textup{nlp}} t\}$ locally uniformly. One can then apply the arguments in \cite[Section 4]{Evans_1989} to show that
$$
\liminf_{\epsilon \to 0} u\left( \frac{t}{\epsilon}, \frac{x}{\epsilon}\right) >0 \text{ in } \{ (t,x): t>0 \, \text{ and }\, x <  {c}_{\textup{nlp}} t\}.
$$
This implies that ${c}_2 \geq {c}_{\textup{nlp}}$ (see Proposition \ref{thm:3.6}).

\item To estimate $c_{2}$ from above, we observe that, for some $\delta^*>0$, $w_* \geq w_1$ in
$\{ (t,x): \,  x \geq (2\sqrt{dr} - \delta^*)t\}$. Hence, together with  \eqref{eq:upperlimitw} we obtain a large deviation estimate of $u$. Namely, for $\hat{c}  = 2\sqrt{dr} - \delta^*$,
$$
u(t, \hat{c} t) \leq \exp \left( -[ \hat \mu + o(1)]t \right) \quad \text{ for } t \gg 1,
$$
where $\hat\mu = w_1(1, \hat{c})$. Now, recalling $(u,v)$ is a solution to \eqref{eq:1-1} restricted to the domain $\{(x,t): 0\leq x \leq \hat{c} t\}$, with boundary condition satisfying
$$
\lim_{t\to\infty} (u,v)(t,0) = (k_1,k_2) \quad \text{ and }\quad \lim_{t\to\infty} (u,v)(t, \hat{c} t) = (0,1),
$$
we may compare, within the domain $\{(x,t): 0\leq x \leq \hat{c} t\}$, the solution $(u,v)$ with suitable traveling wave solutions connecting $(k_1,k_2)$ with $(0,1)$ to control the spreading speed $c_2$ of $u$ from above (Lemma \ref{lem:appen1}).

\end{enumerate}

\subsection{Organization of the paper}

In Section \ref{sec:2}, we determine $c_1, c_3$ and give rough estimates of $c_2$. In Section \ref{S3}, we establish the approximate asymptotic expression of $u$ and then determine ${c}_2$ in Section \ref{sec:4}. This completes the proof of Theorem \ref{thm:1-1}. In Section \ref{sec:5}, Theorem \ref{thm:1-1'} is derived as a limiting case of Theorem \ref{thm:1-1}.
To improve the exposition of ideas, we postpone the proofs of Lemma \ref{lem:appen1} and Proposition \ref{prop:B2} to the Appendix.

\section{Preliminaries}\label{sec:2}

We define the maximal and minimal spreading speeds as follows (see also \cite[Definition 1.2]{Hamel2012} where related concepts were introduced for a single species):
\begin{equation*}
\begin{cases}
\smallskip
\overline{c}_1=\inf{\{c>0~|~\limsup \limits_{t\rightarrow \infty}\sup\limits_{x>ct} v(t,x)=0\}},\\
\smallskip
\underline{c}_1=\sup{\{c>0~~|\liminf \limits_{t\rightarrow \infty}\inf\limits_{ct-1<x<ct} v(t,x)>0\}}, \\
\smallskip
\overline{c}_2=\inf\{c>0~|~\limsup\limits_{t\rightarrow \infty}\sup\limits_{ x> ct}u(t,x)=0\}, \\
\smallskip
\underline{c}_2=\sup{\{c>0~~|\liminf \limits_{t\rightarrow \infty}\inf\limits_{ct-1<x<ct} u(t,x)>0\}},\\
\smallskip
\overline{c}_3=\inf\{c<0~|~\liminf\limits_{t\rightarrow \infty}\inf\limits_{ ct<x< ct+1}v(t,x) >0\},\\
\underline{c}_3=\sup\{c<0~|~\limsup\limits_{t\rightarrow \infty}\sup\limits_{ x< ct}v(t,x)=0\}.
 \end{cases}
\end{equation*}
Here $\overline{c}_1$ and $\underline{c}_1$ (resp.  $\overline{c}_2$ and $\underline{c}_2$) are the maximal and minimal rightward spreading speeds of species $v$ (resp. species $u$), whereas  $-\underline{c}_3$ and $-\overline{c}_3$ are the maximal and minimal leftward spreading speeds of $v$.

In this section, we will determine $\overline{c}_1=\underline{c}_1$ and $\overline{c}_3=\underline{c}_3$, and give some rough estimates of $\underline{c}_2$ and $\overline{c}_2$. We will also show that the solution $(u,v)$ of \eqref{eq:1-1} approaches one of the homogeneous equilibria in between successive spreading speeds. Recalling the definition of $c_1=2\sqrt{dr}$ and $c_3 = -\tilde{c}_{\rm LLW}$ in \eqref{eq:c1c2c3}, the main result of this section can be precisely stated as follows. 

\begin{proposition}\label{prop:1} 
Assume $dr > 1$.
Let $(u,v)$ be the solution of \eqref{eq:1-1} with initial data  $(u_0, v_0)$ satisfying {\rm $\rm{(H_{\infty})}$}.  Then 
\begin{itemize}
\item[{\rm(i)}] $\overline{c}_1 = \underline{c}_1 = c_1$ and $\overline{c}_3 = \underline{c}_3 = c_3$;

\item[{\rm(ii)}]  ${c}_{\textup{LLW}} \leq \underline{c}_2\leq \overline{c}_2\leq 2$;

\item[{\rm(iii)}] For each small $\eta>0$, the following spreading results hold:
\begin{subequations}\label{eq:spreadinglyprop}
\begin{align}
&\lim\limits_{t\to \infty} \sup\limits_{ x>(c_1+\eta) t} (|u(t,x)|+|v(t,x)|)=0, \label{eq:spreadinglyprop_a} \\
&\lim\limits_{t\to \infty} \sup\limits_{(\overline{c}_2+\eta) t< x<(c_1-\eta) t} (|u(t,x)|+|v(t,x)-1|)=0,\label{eq:spreadinglyprop_b} \\
&\lim\limits_{t\to\infty} \sup\limits_{(c_3+\eta)t< x<(\underline{c}_2-\eta) t} (|u(t,x)-k_1|+|v(t,x)-k_2|)=0, \label{eq:spreadinglyprop_c} \\
&\lim\limits_{t\rightarrow \infty} \sup\limits_{x<({c_3}-\eta)t} (|u(t,x)-1|+|v(t,x)|)=0,\label{eq:spreadinglyprop_d} 
\end{align}
\end{subequations}
\end{itemize}
where ${c}_{\textup{LLW}},\, \tilde{c}_{\textup{LLW}}$ are given in Theorem \ref{thm:LLW} and Remark \ref{rmk:LLW}, respectively.
\end{proposition}
\begin{remark}
 Proposition \ref{prop:1} is proved in \cite{Lin_2012} under the stronger assumption $dr(1-b) >1$.
\end{remark}

Before estimating the spreading speeds of species, we first give a lemma concerning the behaviors of $(u,v)$ between the spreading fronts.
\begin{lemma}\label{lem:entire1}
Let $-\infty\leq\underline c<\overline c\leq\infty$ be fixed, and
let $(u,v)$ be a solution of \eqref{eq:1-1} in $\{(t,x): \underline{c} t \leq x \leq \overline{c}t\}$.
\begin{itemize}
\item[{\rm(a)}]If
$\displaystyle \liminf_{t \to \infty} \inf_{(\underline c + \delta) t < x < (\overline c - \delta)t} v(t,x) >0$ for each $0<\delta < (\overline{c}-\underline{c})/2$, then
 $$\displaystyle \limsup_{t \to \infty} \sup_{(\underline c + \delta) t < x < (\overline c - \delta)t}u(t,x) \leq k_1, \quad \liminf_{t \to \infty} \inf_{(\underline c + \delta) t < x < (\overline c - \delta)t}v(t,x) \geq k_2,$$
 for each $0<\delta < (\overline{c}-\underline{c})/2$;

 \item[{\rm(b)}] If $\displaystyle \lim_{t \to \infty} \sup_{(\underline c + \delta) t < x < (\overline c - \delta)t} u(t,x)=0$ and $\displaystyle \liminf_{t \to \infty} \inf_{(\underline c + \delta) t < x < (\overline c - \delta)t} v(t,x) >0$  for each $0 < \delta < (\overline{c}-\underline{c})/2$, then
     $$\displaystyle \lim_{t \to \infty} \sup_{(\underline c + \delta) t < x < (\overline c - \delta)t }|v(t,x) -1|=0, \quad \text{ for each }0 < \delta < (\overline{c}-\underline{c})/2;$$

\item[{\rm(c)}]
If $\displaystyle \liminf_{t \to \infty} \inf_{(\underline c + \delta) t < x < (\overline c - \delta)t} u(t,x) >0$ for each $0 < \delta < (\overline{c}-\underline{c})/2$, then
 $$\displaystyle \liminf_{t \to \infty} \inf_{(\underline c + \delta) t < x < (\overline c - \delta)t}u(t,x) \geq k_1, \quad \limsup_{t \to \infty} \sup_{(\underline c + \delta) t < x < (\overline c - \delta)t}v(t,x) \leq k_2
 $$
 for each $0<\delta < (\overline{c}-\underline{c})/2$;

\item[{\rm(d)}] If $\displaystyle \lim_{t \to \infty} \sup_{(\underline c + \delta) t < x < (\overline c - \delta)t} v(t,x)=0$ and $\displaystyle \liminf_{t \to \infty} \inf_{(\underline c + \delta) t < x < (\overline c - \delta)t} u(t,x) >0$ for each $0<\delta < (\overline{c}-\underline{c})/2$, then
    $$\displaystyle \lim_{t \to \infty} \sup_{(\underline c + \delta) t < x < (\overline c - \delta)t}|u(t,x) -1|=0, \quad \text{ for each }0 < \delta < (\overline{c}-\underline{c})/2.$$

\end{itemize}
\end{lemma}
\begin{proof}
The proof is based on classification of entire solutions of \eqref{eq:1-1}.
For $(x_1,x_2)$ and $(y_1,y_2)$ in $\mathbb{R}^2$, we define the partial order $``\preceq"$ so that
$$(x_1,x_2) \preceq (y_1,y_2) \quad \text{ if and only if }\quad
x_1 \leq x_2 \,\,\text{ and }\,\, y_1 \geq y_2.
$$
 Suppose {\rm(a)} is false. Then there exists $(t_n,x_n)$ such that, as $n \to \infty$,  {$t_n\to\infty$ and}
\begin{equation*}
    \begin{array}{c}
        c_n:= \frac{x_n}{t_n} \to c \in (\underline c, \overline c) \,\text{ and }\, \lim\limits_{n\to\infty}u(t_n,x_n)>k_1 \text{ or }  \lim\limits_{n\to\infty}v(t_n,x_n)<k_2.
     \end{array}
\end{equation*}
Define $(u_n,v_n)(t,x):= (u,v)(t_n + t, x_n + x)$. It is standard to show that $0 \leq u_n \leq 1$ and $0 \leq v_n \leq 1$ in $[-t_n,\infty)\times\mathbb{R}$, so that by parabolic estimates $(u_n,v_n)$ is precompact in $C^2_{\mathrm{loc}}(K)$ for each compact subset $K \subset \mathbb{R}^2$.
Passing to a subsequence, we may assume that $(u_n,v_n)$ converges  to an entire solution
$(\hat u, \hat v)$ of \eqref{eq:1-1} in $C^2_{\mathrm{loc}}(\mathbb{R} \times \mathbb{R})$. By construction, there exists a constant $0<\zeta_0<1$ such that $ (\hat{u}, \hat{v})(t,x) \preceq (1, \zeta_0)$ for $(t,x) \in \mathbb{R}^2$.
Let $(\overline{U}, \underline{V})$ be the solution of the Lotka-Volterra system of ODEs
$$
U_t = U(1-U - aV), \quad V_t = rV(1-bU-V),
$$
with initial data $(1,\zeta_0)$, so that $(\overline{U}, \underline{V})(\infty) = (k_1,k_2)$. Now, for each $T>0$ we have $(\hat{u},\hat{v})(-T,x) \preceq (\overline{U},\underline{V})(0)$ for all $x$, it follows by comparison that
$$
 (\hat u,\hat v)(t,x)\preceq (\overline{U}, \underline{V})(t+T) \, \text{ for }(t,x) \in [-T,0] \times \mathbb{R}\, \text{ and }\,T>0,
$$
so that 
$$
 (\hat u,\hat v)(0,0)\preceq (\overline{U}, \underline{V})(T) \quad \text{ for }T>0.
$$
Letting $T \to \infty$, we obtain $(\hat u,\hat v)(0,0) \preceq (k_1,k_2)$.
In particular, we deduce that
$$
\lim_{n \to \infty} (u,v)(t_n,x_n) = \lim_{n \to \infty} (u_n,v_n)(0,0) = (\hat u, \hat v)(0,0)\preceq (k_1,k_2).
$$
This is a contradiction and  proves {\rm(a)}. The other assertions follow from similar considerations.
\end{proof}

The following lemma says
that the maximal spreading speed of $u$ (resp. $v$) can be estimated by the large deviation estimate of $u$ (resp. $v$) along a line
$\{(t,x): x = \hat{c} t\}$.

\begin{lemma}\label{lem:appen1} 
Let $\hat{c}>0$, $t_0>0$, and  $(\tilde u,\tilde v)$ be a solution of
\begin{equation}\label{eq:A2}
\left\{
\begin{array}{ll}
\medskip
\partial_t \tilde u-\partial_{xx}\tilde u=\tilde u(1-\tilde u-a\tilde v),& 0\leq x\leq\hat ct, t>t_0,\\
\medskip
\partial _t \tilde v-d\partial_{xx}\tilde v=r \tilde v(1-b\tilde u-\tilde v),&  0\leq x\leq\hat ct, t>t_0,\\
\tilde u(t_0,x)=\tilde u_0(x), \tilde v(t_0,x)=\tilde v_0(x), &  0\leq x\leq\hat ct_0.
\end{array}
\right .
\end{equation}
\begin{itemize}
\item[{\rm (a)}] If $\hat{c}>2$ and there exists $\hat \mu>0$ such that 
\begin{itemize}
\item[{\rm(i)}] $\lim\limits_{t\to\infty}(\tilde u,\tilde v)(t,0)=(k_1,k_2)$ and $\lim\limits_{t\to\infty}(\tilde u,\tilde v)(t,\hat ct)=(0,1)$,
\item[{\rm(ii)}]  $\lim_{t\to\infty} e^{\mu t}\tilde u(t,\hat ct)=0~$ for each $\mu \in [0,\hat \mu),$

\end{itemize}
then 
$$
\lim_{t\to\infty} \sup_{ct<x\leq \hat{c}t} \tilde u(t,x) = 0   \quad \text{ for each }c >  c_{\hat{c},\hat\mu},
$$
where

\begin{equation*}
 c_{\hat{c},\hat\mu}=\left\{
\begin{array}{ll}
\medskip
c_{\textup{LLW}},& \text{if } \hat\mu\geq   \lambda_{\textup{LLW}} (\hat{c} - c_{\textup{LLW}}),\\
\hat c-\frac{2\hat\mu}{\hat c-\sqrt{\hat c^2-4(\hat\mu+1-a)}},& \text{if }0< \hat\mu< \lambda_{\textup{LLW}}(\hat{c} - c_{\textup{LLW}}); \\
\end{array}
\right .
\end{equation*}
\item[{\rm (b)}] If $\hat{c} > 2\sqrt{dr}$ and there exists $\hat\mu >0$ such that
\begin{itemize}

\item[{\rm(i)}] $\lim\limits_{t\to\infty}(\tilde u,\tilde v)(t,0)=(k_1,k_2)$, and $\lim\limits_{t\to\infty}(\tilde u,\tilde v)(t,\hat ct)=(1,0)$,

\item[{\rm(ii)}]  $\lim_{t\to\infty} e^{\mu t}\tilde v(t,\hat ct)=0~$ for each $\mu \in [0, \hat \mu),$

\end{itemize}
then 
$$
\lim_{t\to\infty} \sup_{ct<x\leq \hat{c}t} \tilde v(t,x) = 0   \quad \text{ for each }c >\tilde{c}_{\hat{c},\hat\mu},
$$
where
\begin{equation*}
\tilde{c}_{\hat{c},\hat\mu}= \left\{
\begin{array}{ll}
\medskip
\tilde{c}_{\textup{LLW}},& \text{ if } \hat\mu\geq \tilde{\lambda}_{\rm LLW} (\hat{c} - \tilde{c}_{\textup{LLW}}),\\
\hat c-\frac{2d\hat\mu}{\hat c-\sqrt{\hat c^2-4d[\hat\mu+r(1-b)]}},&  \text{ if } 0<\hat\mu<\tilde{\lambda}_{\rm LLW} (\hat{c} - \tilde{c}_{\textup{LLW}}).
\end{array}
\right.
\end{equation*}
\end{itemize}
Here  ${c}_{\textup{LLW}}, \tilde{c}_{\textup{LLW}}$ are given in Theorem \ref{thm:LLW} and Remark \ref{rmk:LLW}, and
\begin{equation}\label{eq:LLL}
\lambda_{\rm LLW}=\frac{{c}_{\textup{LLW}}-\sqrt{{c}_{\textup{LLW}}^2-4(1-a)}}{2}, \quad  \tilde{\lambda}_{\rm LLW}=\frac{\tilde{c}_{\textup{LLW}}-\sqrt{\tilde{c}_{\textup{LLW}}^2-4dr(1-b)}}{2d}.\end{equation}
\end{lemma}
The proof of Lemma \ref{lem:appen1} is based on comparison with {appropriate} traveling wave solutions connecting $(k_1,k_2)$ with one of the semi-trivial equilibrium points. We postpone the proof to Appendix \ref{sec:B}.

\begin{proof}[Proof of Proposition \ref{prop:1}]
 It follows directly from definition that $\underline{c}_i \leq \overline{c}_i$ for $i=1,2,3$. We will complete the proof in the following order: \rm{(1)} $\overline{c}_2 \leq 2$,  \rm{(2)} $\overline{c}_1 \leq 2\sqrt{dr}$,  \rm{(3)} $\overline c_3\leq -\tilde{c}_{\rm{LLW}}$,  \rm{(4)}  $\underline c_2 \geq c_{\rm LLW}$,  \rm{(5)} $\underline c_1\geq 2\sqrt{dr}$,  \rm{(6)} $\underline{c}_3 \geq -\tilde{c}_{\rm{LLW}}$. After that, we establish  \eqref{eq:spreadinglyprop_a}-\eqref{eq:spreadinglyprop_d} by applying Lemma \ref{lem:entire1}. Our proof adapts the ideas of \cite{Ducrot_preprint} and \cite[Proposition 3.1]{Girardin_2018}, and can be skipped by the motivated reader.

\noindent {\bf Step 1.} We show assertions (1) and (2).

Fix $\lambda >0$, let $A \gg 1$ be chosen such that $\overline{u}_\lambda(t,x):= \exp(-\lambda x+(\lambda^2 + 1)t+A)$ satisfies
\begin{equation*}\label{eq:uukpp}
\left\{
\begin{array}{ll}
\partial_t \overline{u}_\lambda - \partial_{xx} \overline{u}_\lambda \geq \overline{u}_\lambda(1-\overline{u}_\lambda), &\text{ in }(0,\infty) \times \mathbb{R},\\
\overline{u}_\lambda(x,0) \geq u_0(x), &\text{ for } x \in \mathbb{R}.
\end{array}
\right.
\end{equation*}
By comparison principle, we have
\begin{equation}\label{eq:esta}
0 \leq u(t,x) \leq \exp(-\lambda x+(\lambda^2 + 1)t+A).
\end{equation} Setting $\lambda=1$, we have
\begin{equation}\label{eq:ukpp}
\lim_{t \to \infty} \sup_{x > (2+\eta)t} |u(t,x)| = \lim_{t \to \infty} \sup_{x > (2+\eta)t} \exp(-x + 2t + A)= 0 \, \text{ for each }\eta >0.
\end{equation}
Thus $\overline c_2\leq 2$, i.e., assertion  \rm{(1)}  holds. Similarly, we have for each $\eta >0$,
\begin{equation}\label{eq:vkpp}
\lim_{t \to \infty} \sup_{|x| > (2\sqrt{dr}+\eta)t} |v(t,x)| = 0,
\end{equation}
i.e.  $\overline c_1\leq 2\sqrt{dr}$ and assertion \rm{(2)} holds. In addition, we deduce (\ref{eq:spreadinglyprop_a}) as $dr>1$.

\noindent {\bf Step 2.} We show assertion (3), i.e.,  $\overline{c}_3 \leq - \tilde{c}_{\rm{LLW}}$.

By ${\rm (H_\infty)}$, ${v}_0$ is non-trivial, compactly supported and 
$$
(u_0(x),v_0(x)) \preceq (1, {v}_0(x)) \quad \text{ in }\mathbb{R}.
$$
Let $(\tilde u_{\rm LLW}, \tilde v_{\rm LLW})$ be the solution to \eqref{eq:1-1} with initial condition $(1, {v}_{0}(x))$. Then
 Remark \ref{rmk:LLW} guarantees the existence of
$\tilde{c}_{\rm{LLW}} \in [2\sqrt{dr(1-b)}, 2\sqrt{dr}]$, such that 
$$
\liminf_{t \to \infty} \inf_{|x| < |c|t } \tilde v_{\rm LLW}(t,x)  >0, \quad \text{ for each } c \in (-\tilde{c}_{\rm{LLW}}, 0).
$$
By the comparison principle for \eqref{eq:1-1},
$(u,v) \preceq (\tilde u_{\rm LLW},\tilde v_{\rm LLW})$ for all $(t,x) \in (0,\infty) \times \mathbb{R}$, which yields, for each $c \in (- \tilde{c}_{\rm{LLW}}, 0)$,
$$
\liminf_{t \to \infty} \inf_{ct < x < ct + 1} v(t,x) \geq  \liminf_{t \to \infty} \inf_{ct < x < ct + 1} \tilde v_{\rm LLW}(t,x)  >0
$$
This proves $\overline{c}_3 \leq - \tilde{c}_{\rm{LLW}}\leq -2\sqrt{dr(1-b)}$ and thus assertion \rm{(3)} holds.

\noindent {\bf Step 3.} We show assertion (4), i.e., $\underline{c}_2 \geq c_{\rm LLW}$.

As in Step 2, this can be proved 
by comparing $(u,v)$ with the solution $(u_{\rm LLW}, v_{\rm LLW})$ of \eqref{eq:1-1} with initial condition $(\tilde{u}_0, 1)$, for some compactly supported $\tilde{u}_0$ satisfying $0 \leq \tilde{u}_0 \leq u_0$, and then using  Theorem \ref{thm:LLW}. 

\noindent {\bf Step 4.} We show assertion (5), i.e., $\underline{c}_1 \geq 2\sqrt{dr}$.

 Fix $c \in (2, 2\sqrt{dr})$, and choose $\eta_1>0$ small enough so that
\begin{equation*}
[c-\eta_1, c+\eta_1]\subset (2, 2\sqrt{dr}), \quad \frac{c^2}{4d} - r(1-2\eta_1) + \frac{d \pi^2 \eta^2_1}{4}<0, 
\end{equation*}
and then choose, by \eqref{eq:ukpp},  $T_1 >1/\eta_1^2$ large enough so that
\begin{equation*}
au(t,x) \leq \eta_1 \,\, \text{ in } \Omega_1,
\end{equation*}
where $\Omega_1 = \{(t,x): (c-\eta_1)t \leq  x \leq  (c+ \eta_1)t,\, t \geq T_1\}.$
Now, let $\eta_2 \in (0,\eta_1]$, and define
\begin{equation*}
\underline{v}^{c}(t,x):=\begin{cases}
\eta_2 e^{-\frac{c}{2d}(x-c t + 1/\eta_1)}\cos\left(\frac{\eta_1\pi(x-ct)}{2}\right), & \text{ if } |x-ct| < 1/\eta_1,\\
0 &\text{ if } |x-ct| \geq 1/\eta_1,
\end{cases}
\end{equation*}
where $\eta_2$ is chosen small enough to ensure that $\underline{v}^{c}(t,x)\leq v(t,x)$ on the parabolic boundary of $\Omega_1$.

It can be verified that $v(t,x)$ and $\underline{v}^{c}(t,x)$
are respectively super- and sub-solutions of the equation
$$
\partial_t \tilde v - d \partial_{xx} \tilde v = r\tilde v( 1- \eta_1 - \tilde v) \quad \text{ in } \Omega_1.
$$
 By the comparison principle, we deduce that
$$
\liminf_{ t\to\infty} v(t, ct) \geq \liminf_{ t\to\infty} \underline{v}^{c}(t, ct)>0.
$$
Hence, $\underline{c}_1 \geq c$. Letting $c \nearrow 2\sqrt{dr}$, we have $\underline{c}_1 \geq 2\sqrt{dr}$.

\noindent {\bf Step 5.} We claim that
\begin{equation}\label{positivev}
    \displaystyle \liminf_{t \to \infty} \inf_{(\overline{c}_3 + \eta)t < x < (\underline{c}_1 - \eta)t} v(t,x) >0 \text{ for  small } \eta>0.
\end{equation}

Given any small $\eta>0$, definitions of $\overline{c}_3$ and $\underline{c}_1$ imply the existence of $c_3' \in (\overline{c}_3, \overline{c}_3 + \eta)$, $c_1' \in (\underline{c}_1 - \eta, \underline{c}_1)$ and $T>0$ such that
$$
\inf_{t \geq T}\min\{v(t, {c}_3' t), v(t, {c}'_1 t)\} >0.
$$
Now, define
\begin{equation*}
    \begin{array}{l}
     \delta:= \min\left\{\frac{1-b}{2},\, \inf\limits_{c_3' T < x < c'_1 T}v(T,x),\,\inf\limits_{t \geq T}\min\{v(t, {c}_3' t), v(t, {c}'_1 t)\} \right\} >0.
     \end{array}
\end{equation*}
Observe that $v(t,x)$ is a super-solution to the KPP-type equation $\partial_t v = d\partial_{xx}v + rv(1-b - v)$ satisfying $v(t,x) \geq \delta$ on the parabolic boundary of the domain $\Omega:=\{(t,x): t \geq T,\,c_3' t < x < c'_1 t \}$. Since $v - \delta$ cannot attain interior negative minimum, it follows that
$v \geq \delta$ in $\Omega$. In particular, \eqref{positivev} holds.

\noindent {\bf Step 6.} We show that
\begin{equation}\label{positiveu}
    \displaystyle \liminf_{t \to \infty} \inf_{x <(\underline{c}_2-\eta)t} u(t,x) >0\text{ for small }\eta>0.
\end{equation}
Fix a small $\eta>0$.
By definition of $\underline{c}_2>0$, there exists $c'_2 \in (\underline{c}_2 - \eta, \underline{c}_2)$ and $T_2>0$ such that
\begin{equation}\label{eq:c2'}
\inf_{t \geq T_2} u(t,c'_2 t) >0.
\end{equation}
 Observe also that $v \leq 1$ and thus  $u$ is a super-solution to
$$
\left\{
\begin{array}{ll}
\partial_t u = \partial_{xx}u + u(1-a - u), &\text{ for }(x,t) \in \mathbb{R} \times (0,\infty),\\
u(x,0) = \chi_{(-\infty,0]} \theta_0, &\text{ for }x \in \mathbb{R},
\end{array}
\right.
$$
where $\theta_0>0$ is given by ${\rm (H_\infty)}$. It follows from the classical results in \cite{Fisher_1937,Kolmogorov_1937} that, for some $T_2>0$,
\begin{equation}\label{eq:ukpp1}
\inf_{t \geq T_2} \inf_{x < (2\sqrt{1-a}-\eta)t} u(t,x) \geq  \frac{1-a}{2}>0 \, \text{ for all }\eta >0.
\end{equation}
Since $\{(T_2,x):\, (2\sqrt{1-a}-\eta)T_2 \leq x \leq c'_2 T_2\}$ is a compact set, \eqref{eq:ukpp1} implies
\begin{equation}\label{eq:ukppp1}
\inf_{x \leq c'_2 T_2} u(T_2,x) >0.
\end{equation}
By \eqref{eq:c2'} and \eqref{eq:ukppp1}, we deduce that $\delta := \min\{\inf\limits_{t \geq T_2}  u(t,c'_2 t),\, \frac{1-a}{2},\, \inf_{x \leq c'_2 T_2} u(T_2,x) \}$ is positive, then $u$ is a super-solution to the KPP-type equation
$\partial_t  u= \partial_{xx} u + u(1-a-u)$ in the domain $\Omega':=\{(t,x): t \geq T_2,\,  x \leq c'_2 t\}$
such that $u(t,x) \geq \delta$ on the parabolic boundary. Therefore, we deduce  $u(t,x) \geq \delta$ in $\Omega'$ and \eqref{positiveu} follows.

\noindent{\bf Step 7. }  We show (\ref{eq:spreadinglyprop_b}) and (\ref{eq:spreadinglyprop_c}).

Fix small $\eta>0$.  Since \eqref{positivev} holds, and $\lim\limits_{t\to \infty } \sup\limits_{x>(\overline{c}_2+\eta)t} u=0$ (by definition of $\overline{c}_2$),
we may apply  Lemma \ref{lem:entire1}(b) to deduce
(\ref{eq:spreadinglyprop_b}).

Next, in view of \eqref{positivev} and \eqref{positiveu} and the fact that $\overline{c}_3\leq -\tilde{c}_{\rm LLW}$, one can deduce  (\ref{eq:spreadinglyprop_c}) from
 items \rm{(a)} and \rm{(c)} of Lemma \ref{lem:entire1}.

It remains to show $\underline{c}_3\geq -\tilde{c}_{\rm LLW}$. 

\noindent{\bf Step 8. } We claim
\begin{equation}\label{step8}
    \displaystyle \lim_{t \to \infty} \sup_{x < (-2\sqrt{dr} - \eta)t} |u(t,x)  -1| = 0 \text{ for each }\eta>0.
\end{equation}

Observe from \eqref{eq:vkpp} and \eqref{eq:ukpp1} that for each $\eta>0$,
$$\displaystyle \lim_{t\to \infty}\sup_{x<(-2\sqrt{dr}-\eta)t} |v(t,x)|=0, \quad \text{ and} \quad \liminf_{t \to \infty} \inf_{x < (-2\sqrt{dr}-\eta)t} u(t,x) \geq  \frac{1-a}{2}.$$
Thus \eqref{step8} follows by applying Lemma  \ref{lem:entire1}\rm{(d)}.

\noindent{\bf Step 9. }  We claim that, for each $\lambda>0$, there exists $K>0$ such that
\begin{equation}\label{eq:step9}
v(t,x) \leq \min\left\{1, \exp(\lambda (x + K) + (d\lambda^2 + r)t)\right\}.\end{equation}

To this end, choose $K>0$ such that $v_0(x) \leq \chi_{[-K,\infty)}$, then the right hand side of \eqref{eq:step9} defines a weak super-solution
 of the KPP-type equation $\partial_t v = d\partial_{xx} v + rv(1-v)$.

\noindent{\bf Step 10. } We finally  show $\underline{c}_3\geq -\tilde{c}_{\rm{LLW}}$ and establish (\ref{eq:spreadinglyprop_d}).

We first apply  Lemma \ref{lem:appen1} to show $\underline{c}_3\geq -\tilde{c}_{\rm{LLW}}$. Let $\tilde{u}(t,x)=u(t,-x)$ and $\tilde{v}(t,x)=v(t,-x)$ and let $\hat{c} >2\sqrt{dr}$ be a constant to be specified later. Recalling (\ref{eq:spreadinglyprop_c}) proved in Step 7 and \eqref{step8}, we arrive at
$$
\lim\limits_{t\to \infty}(\tilde{u},\tilde{v})(t,0)=(k_1,k_2) \text{ and }\lim\limits_{t\to \infty}(\tilde{u},\tilde{v})(t,\hat{c}t)=(1,0).
$$
This verifies hypotheses (i) and (ii) of Lemma \ref{lem:appen1}\rm{(b)}. Next, by \eqref{eq:step9}, we have for arbitrary $\lambda>0$,
$$\tilde{v}(t,\hat{c}t)= v(t,-\hat{c}t)\leq \exp\{-\mu^\lambda t+\lambda K\} =  \exp\{-(\mu^\lambda + o(1)) t\},$$
where $\mu^\lambda=\hat{c}\lambda-d\lambda^2-r$. To {apply} Lemma \ref{lem:appen1}\rm{(b)}, we need to choose $\lambda$ and $\hat{c}$ such that
\begin{equation}\label{eq:mumumu}
\mu^{\lambda} \geq \tilde{\lambda}_{\rm LLW} (\hat{c} - \tilde{c}_{\rm{LLW}})  =-d\tilde{\lambda}^2_{\rm LLW}+\tilde{\lambda}_{\rm LLW}\hat c-r(1-b),
\end{equation}
where the {equality} follows from definition of $\tilde{\lambda}_{\rm LLW}$ in \eqref{eq:LLL}.
Observing
\begin{equation*}
   \mu^{\lambda}-\left[-d\tilde{\lambda}^2_{\rm LLW}+\tilde{\lambda}_{\rm LLW}\hat c-r(1-b)\right]= (\lambda-\tilde{\lambda}_{\rm LLW})\left[\hat c- d(\lambda+\tilde{\lambda}_{\rm LLW})\right]-rb,
\end{equation*}
we may fix $\lambda > \tilde{\lambda}_{\rm LLW}$ and choose $\hat{c}$ large enough so that \eqref{eq:mumumu} is verified. Now,  applying Lemma \ref{lem:appen1}\rm{(b)} to $(\tilde{u},\tilde{v})$, we conclude that for any $c>\tilde{c}_{\hat{c},\mu}=\tilde{c}_{\rm{LLW}}$,
\begin{equation*}
\begin{array}{l}
\lim\limits_{t\rightarrow \infty}\sup\limits_{ x> ct}\tilde{v}(t,x)=\lim\limits_{t\rightarrow \infty}\sup\limits_{ x< -ct}v(t,x)=0.
\end{array}
\end{equation*}
 This implies $\underline{c}_3\geq -\tilde{c}_{\rm{LLW}}$.

 Furthermore,  in view of \eqref{positivev},
we can deduce (\ref{eq:spreadinglyprop_d}) by Lemma \ref{lem:entire1}\rm{(d)}.
The proof of Proposition \ref{prop:1} is now complete.
\end{proof}

\begin{remark}\label{rem_2.1}
By Steps 2 and 3 in the  proof of Proposition \ref{prop:1}, observe that the assertions {\rm{(3)}} $\overline c_3\leq -\tilde{c}_{\textup{LLW}}$ and { \rm{(4)}}  $\underline c_2 \geq c_{\rm LLW}$ remain true for more general initial data $(u_0,v_0)$, e.g., when
$$
\liminf_{x \to -\infty} u_0(x) >0,\quad \lim_{x \to \infty} u_0(x) = 0, \quad \text{ and } \quad \lim_{|x| \to \infty} v_0(x) = 0.
$$
\end{remark}

\section{Estimating $\underline{c}_2$ and $\overline {c}_2$ via geometric optics ideas}\label{S3}
Throughout this section, we assume that there exists $c_1 > \tilde{c}_1 \geq 2$ such that
\begin{equation}\label{eq:v_ep}
\chi_{\{\tilde{c}_1 t < x < c_1 t\}}\leq\hspace{-.3cm}\liminf_{\scriptsize \begin{array}{c}\epsilon \to 0\\ (t',x') \to (t,x)\end{array}} \hspace{-.3cm} v^\epsilon(t',x') \leq\hspace{-.3cm}\limsup_{\scriptsize \begin{array}{c}\epsilon \to 0\\ (t',x') \to (t,x)\end{array}} \hspace{-.3cm}  v^\epsilon(t',x')\leq \chi_{\{x \leq c_1 t\}},
\end{equation}
{for all $(t,x)\in(0,\infty)\times\mathbb{R}$.}
\begin{remark}
Under the assumptions $dr>1$ and ${\rm (H_\infty)}$, the condition \eqref{eq:v_ep} holds for
 $\tilde{c}_1 = 2$ and $c_1 = 2\sqrt{dr}$, by invoking Proposition \ref{prop:1}.
 \end{remark}

To prove Theorem \ref{thm:1-1}, it remains to deduce $\underline c_2=\overline c_2$ and determine its value. In view of Lemma \ref{lem:appen1}, the key is  to choose $\hat{c}>2$ and determine $\hat \mu>0$ such that
\begin{equation}\label{eq:exppp}
u(t, \hat{c} t)= \exp\left(  - (\hat \mu + o(1))t\right).
\end{equation}
This was accomplished in \cite{Girardin_2018} for the case $a < 1 < b$ by a delicate construction of global super- and sub-solutions, in the sense that they are defined and respect the differential inequalities for $(t,x) \in (0,\infty)\times \mathbb{R}$.

In this section, we shall derive the exponential estimate \eqref{eq:exppp} by the ideas of large deviations. Using this method, one can obtain {an} exponential estimate of $u$ without constructing global super- and sub-solutions for system \eqref{eq:1-1}.

To this end, we introduce a small parameter $\epsilon$ via the following transformation:
\begin{equation*}
u^\epsilon(t,x)=u\left(\frac{t}{\epsilon},\frac{x}{\epsilon}\right) \quad \mathrm{and} \quad v^\epsilon(t,x)=v\left(\frac{t}{\epsilon},\frac{x}{\epsilon}\right).
\end{equation*}
Under the new scaling, we rewrite  the equation of $u$ in \eqref{eq:1-1} as
\begin{align}\label{eq:1-1'}
\left \{
\begin{array}{ll}
\medskip
\partial_t u^\epsilon=\epsilon \partial_{xx} u^\epsilon+\frac{u^{\epsilon}}{\epsilon}(1-u^\epsilon-av^\epsilon), & \mathrm{in}~ (0,\infty)\times\mathbb{R},\\
u^\epsilon(0,x)=u_0(\frac{x}{\epsilon}), & \mathrm{on}~ \mathbb{R}.\\
\end{array}
\right.
\end{align}

To obtain the asymptotic behavior of $u^\epsilon$ as $\epsilon\rightarrow 0$, the idea is to consider the WKB-transformation $w^\epsilon$, which is given by
\begin{equation}\label{eq:w}
w^\epsilon(t,x)=-\epsilon\log{u^\epsilon(t,x)},
\end{equation}
and satisfies the equation:
\begin{align}\label{eq:3-2}
\left \{
\begin{array}{ll}
\medskip
\partial_tw^\epsilon-\epsilon\partial_{xx} w^\epsilon+|\partial_x w^\epsilon|^2+1-u^\epsilon-av^\epsilon=0, & \mathrm{in}~ (0,\infty)\times\mathbb{R},\\
w^\epsilon(0,x)=-\epsilon\log{u^\epsilon(0,x)},  & \mathrm{on} ~\mathbb{R}.\\
\end{array}
\right.
\end{align}

\begin{lemma}\label{lem:3-1}
Let $w^\epsilon$ be a solution of \eqref{eq:3-2}. Then  for each  compact subset $Q$ of $[(0,\infty)\times \mathbb{R}]\cup [\{0\} \times (-\infty,0)]$, there is a constant $C(Q)$  independent of $\epsilon $ such that
$$0 \leq {w}^{\epsilon}(t,x)
\leq C(Q) \quad \text{ for }\, (t,x) \in Q  \text{ and }\,\, \epsilon \in (0, 1/C(Q)].
$$
Furthermore,
\end{lemma}
\begin{proof}
Since $u^\epsilon \leq 1$, we have $w^\epsilon \geq 0$ by definition. It remains to show the upper bound.
We follow the ideas in \cite[Lemma 2.1]{Evans_1989} to construct suitable super-solutions and apply the comparison principle to derive the desired result. First, fix $\delta \in (0,1)$ such that
$$
Q \subset ([0,1/\delta] \times (-\infty, -\delta]) \cup ([\delta, 1/\delta] \times [-\delta, 1/\delta]).
$$
We will estimate $w^\epsilon$ on $[0,1/\delta] \times (-\infty, -\delta]$ and $[\delta, 1/\delta] \times [-\delta, 1/\delta]$ separately.

Define $Q_0:= (0,\infty) \times (-\infty, 0)$, and  for $\epsilon \in (0,1/2]$,
$$
\overline{z}^\epsilon_1(t,x):=\frac{2\epsilon}{|x|}+2t-\epsilon\log\theta_0 \quad \text{ in }\, Q_0,
$$
where $\theta_0$ is specified in ${\rm (H_\infty)}$.
We claim that for $\epsilon \in (0,1/2]$,
\begin{equation}\label{eq:c_delta}
w^\epsilon(t,x) \leq \overline{z}_1^\epsilon(t,x) \quad \text{ in }\,Q_0.
\end{equation}
To this end, first observe that $\overline{z}_1^\epsilon$ is a (classical) super-solution of \eqref{eq:3-2} in $ Q_0=(0,\infty)\times (-\infty, 0)$. Indeed, for $\epsilon \in (0,1/2]$, $w^\epsilon \leq \overline{z}^\epsilon$ on $\partial Q_0$, and
\begin{align*}
&\partial_t \overline{z}_1^\epsilon - \epsilon \partial_{xx} \overline{z}_1^\epsilon + |\partial_x \overline{z}_1^\epsilon|^2 + 1 - u^\epsilon - av^\epsilon\\
&= 2 + \frac{4\epsilon^2}{|x|^3 }\left( \frac{1}{|x|} - 1\right) + 1 - u^\epsilon - av^\epsilon\\
&\geq 1 +\frac{4\epsilon^2}{|x |^3} \left( \frac{1}{|x|} - 1\right) \geq 0.
\end{align*}
By maximum principle, \eqref{eq:c_delta} holds. This proves, for $\epsilon \in (0,1/2]$,
\begin{equation}\label{eq:c_delta2}
w^\epsilon(t,x) \leq C_\delta \quad \text{ in }\,[0,1/\delta] \times (-\infty, -\delta],
\end{equation}
by taking
 $C_\delta = \sup_{0 <\epsilon \leq 1/2} \sup_{[0,1/\delta] \times (-\infty, -\delta]} \overline{z}^\epsilon_1(t,x)$.

It remains to show, for $\epsilon \in (0,2\delta]$, the uniform boundedness of $w^\epsilon$ in $[\delta, 1/\delta]\times [-\delta, 1/\delta]$. To this end, define
\begin{equation*}
\overline{z}_2^\epsilon(t,x)=\frac{|x+2\delta|^2}{4t}+  t+ \frac{\epsilon}{2} \log t+C_\delta,
\end{equation*}
where $C_\delta>0$ is given in \eqref{eq:c_delta2}. Then
$\overline{z}_2^\epsilon$ is a (classical) super-solution of \eqref{eq:3-2} in $(0,\infty)\times (-\delta,\infty)$.

Moreover,  for each $\tau>0$, $w^\epsilon(\tau,x)$ is finite for all $x \in \mathbb{R}$. Since
$$
\begin{cases}
w^\epsilon(\tau,x) <\infty = \overline{z}_2^\epsilon(0,x) \quad \text{ for }x \geq -\delta,\\
w^\epsilon(t+\tau,-\delta) \leq C_\delta \leq  \overline{z}_2^\epsilon(t,-\delta) \quad \text{ for }t \in  [0,1/\delta-\tau],
\end{cases}
$$
we obtain by comparison that
$$
w^\epsilon(t+\tau, x) \leq \overline{z}_2^\epsilon(t,x)  \quad \text{ for } (t,x) \in [0,1/\delta-\tau] \times [-\delta,\infty).
$$
Letting $\tau \to 0$, we show that
$$
\sup_{[\delta,1/\delta] \times [-\delta,1/\delta]}  w^\epsilon(t, x)  \leq \sup_{[\delta,1/\delta] \times [-\delta,1/\delta]}  \overline{z}_2^\epsilon(t,x) <\infty.
$$
This completes the proof of the local bounds of $w^\epsilon$.
\end{proof}

Having established the $C_{loc}$ bounds, we will pass to the (upper and lower) limits of $w^\epsilon$ by using the half-relaxed limit method, {which is} due to  Barles and Perthame \cite{BP1987}. We begin with the following definition: 
\begin{equation}\label{eq:w_star}
w^*(t,x)=\hspace{-.3cm}\limsup\limits_{\scriptsize \begin{array}{c}\epsilon \to 0\\ (t',x') \to (t,x)\end{array}} \hspace{-.3cm} w^\epsilon (t',x')\quad \mathrm{and} \quad w_*(t,x)=\hspace{-.3cm}\liminf\limits_{\scriptsize \begin{array}{c}\epsilon \to 0\\ (t',x') \to (t,x)\end{array}} \hspace{-.3cm}w^\epsilon (t',x').
\end{equation}
\begin{remark}\label{rmk:Qdelta}
By \eqref{eq:c_delta}, it follows that for any $\delta\in(0,1)$ and $\epsilon$ small,
\begin{equation*}
  \sup_{[0,1/\delta] \times(-\infty,-\delta]} w^\epsilon(t,x) \leq  \sup_{[0,1/\delta] \times(-\infty,-\delta]} \left[\frac{2\epsilon}{|x|}+2t-\epsilon\log\theta_0\right].
\end{equation*}
Sending first $\epsilon \to 0$ then $\delta\to 0$, we deduce $w_*(0,x)=w^*(0,x) = 0$ for all $x \leq 0$.
\end{remark}

\begin{lemma}\label{lem:3-2}
Assume that \eqref{eq:v_ep} holds for some  $c_1 >\tilde{c}_1 \geq 2$. Then
\begin{itemize}
 \item[{\rm(i)}] { $w^*$ is upper semicontinuous and is a viscosity  sub-solution of
\begin{align}\label{eq:3-3}
\left \{
\begin{array}{ll}
\medskip
\min\{\partial_t w+H_1(t,x,\partial_x w) ,w\}=0,&\text{in}~ (0,\infty)\times\mathbb{R},\\
w(0,x)=\left\{\begin{array}{ll} 0,& \text{ for } x\in(-\infty,0],\\
\infty, &\text{ for } x\in(0,\infty);
\end{array}
\right.
\end{array}
\right.
\end{align}}
\item[{\rm(ii)}]{ $w_*$ is lower semicontinuous and is a viscosity super-solution of
\begin{align}\label{eq:3-4}
\left\{
\begin{array}{ll}
\medskip
\min\{\partial_tw+H_2(t,x,\partial_x w),w\}=0,&\text{in}~ (0,\infty)\times\mathbb{R},\\
w(0,x)=\left\{\begin{array}{ll} 0, &\text{ for } x\in(-\infty,0],\\
\infty, &\text{ for } x\in(0,\infty),
\end{array}
\right.
\end{array}
\right.
\end{align}}
\end{itemize}
where
\begin{equation}\label{eq:hamiltonian}
H_1(t,x,p) = |p|^2+1-a\chi_{\{x\leq c_1t\}}  ~\text{ and } ~ H_2(t,x,p) =|p|^2+1-a\chi_{\{\tilde c_1t<x<c_1t\}}.
\end{equation}
\end{lemma}
\begin{proof}
By construction, $w^*, w_*$ are respectively upper and lower semicontinuous. By arguments similar to \cite[Lemma 2.2]{Evans_1989}, one can verify that $w^*, w_*$ are respectively the sub- and super-solutions of the Hamilton-Jacobi equations in $(0,\infty) \times \mathbb{R}$. It remains to check the initial conditions. By Remark \ref{rmk:Qdelta}, we have $w_*(0,x) = w^*(0,x)= 0$ for $x \leq 0$. It remains to check that $w_*(0,x) = \infty$ for $x >0$. For this, we use \eqref{eq:esta} to obtain, for each $\lambda>0$,
$$
w_\epsilon(t',x') \geq \lambda x' - (\lambda^2 + 1) t' + \epsilon A  \quad \text{ for } (t',x') \in (0,\infty)\times \mathbb{R}.
$$
Taking limit inferior as $\epsilon \to 0$ and $(t',x') \to (0,x)$, we have (still for each $\lambda>0$) $w_*(0,x) \geq \lambda x$ for $x >0$. Letting $\lambda \to \infty$, we deduce $w_*(0,x) = \infty$ for all $x >0$. 
\end{proof}


To study the limits $w^*$ and $w_*$ of $w^\epsilon$, we introduce the auxiliary functions $w_i$ ($i=1,2$) as follows. For $(t,x) \in (0,\infty)\times\mathbb{R}$, set $\mathbb{X} = H^1_{\mathrm{loc}}([0,\infty))$ and
$$
\mathbb{X}^{t,x}=\left \{\gamma \in \mathbb{X}\,:\, \gamma(0) = x,\, \gamma(t) \leq   0\right\}.
$$
A mapping $\vartheta: \mathbb{X} \to [0,\infty]$ is a {\it stopping time} provided that for all $\gamma, \hat\gamma \in \mathbb{X}$ and all $s \geq 0$:
\begin{equation}\label{stoppingtime}
  \left\{
\begin{array}{l}
\text{ if }\gamma(\tau) = \hat\gamma(\tau) \text{ for }\tau \in [0,s]\text{ and }\vartheta[\gamma(\cdot)]\leq s,\\
\text{ then }\, \vartheta[\gamma(\cdot)] = \vartheta[\hat\gamma(\cdot)].
\end{array}
\right.
\end{equation}
Let $S$ be an open set in $\mathbb{R}$ and $\gamma \in \mathbb{X}^{t,x}$. An example of stopping time is the first exit time $\tau$ 
from $S$, given by $\tau=\inf\left\{s \in [0,\infty):\gamma(s)\notin S\right\}.$

Denote by $\Theta$ the set of all stopping times. Then for $i=1,2$, $(t,x) \in (0,\infty)\times\mathbb{R}$, we define 
\begin{equation}\label{eq:w1}
w_i(t,x) = \sup_{\vartheta \in \Theta} \left[\inf_{\gamma(\cdot) \in \mathbb{X}^{t,x}} \int_0^{t \wedge \vartheta[\gamma(\cdot)]}L_i({t-s},\gamma(s),\dot\gamma(s))\,ds\right]
\end{equation}
and
\begin{equation}\label{eq:B3-1'}
J_i(t,x) = \inf_{\gamma(\cdot)\in\mathbb{X}^{t,x}} \int_0^t L_i(t-s,\gamma(s),\dot\gamma(s))\,ds,
\end{equation}
where for $i=1,\,2$, $L_i(t,x,q)$ is the Legendre transformation of $H_i(t,x,p)$, 
i.e., $\displaystyle L_i(t,x,q) = \max_{p\in\mathbb{R}} \left[q \cdot p - H_i(t,x,q)\right]$. Precisely,
\begin{equation}\label{eq:legendre}
L_1(t,x,q) =  \frac{|q|^2}{4} - 1 + a \chi_{\left\{x \leq c_1 t\right\}}, \quad 
L_2(t,x,q) = \frac{|q|^2}{4} - 1 + a \chi_{\left\{\tilde{c}_1 t <x < c_1 t\right\}}.
\end{equation}
We state the following calculus lemma, whose proof is postponed to Appendix \ref{sec:A}.
\begin{proposition}\label{prop:B2}
Assume that $c_1 >\tilde{c}_1 \geq 2$. 
Then
 \begin{itemize}

\item[{\rm (a)}] $J_1$ can be expressed as follows.
\begin{equation}\label{eq:propA1}
J_1(t,x)=\left \{
\begin{array}{ll}
\medskip
\frac{t}{4}(\frac{x^2}{t^2}-4), &  \text{ for } \frac{x}{t}\geq c_1\\
\medskip
\left( \frac{c_1}{2}-\sqrt{a}\right)[x-\bar{c}_{\textup{nlp}}t], &\text{ for } c_1-2\sqrt{a}\leq \frac{x}{t}<c_1,\\
\medskip
\frac{t}{4}(\frac{x^2}{t^2}-4(1-a)), &\text{ for } 0\leq\frac{x}{t}<c_1-2\sqrt{a},\\
\medskip
-t(1-a), &\text{ for } \frac{x}{t}<0,
\end{array}
\right.
\end{equation}
where
$\bar{c}_{\textup{nlp}}= \frac{c_1}{2}-\sqrt{a}+\frac{1-a}{\frac{c_1}{2}-\sqrt{a}}$; 
\item[{\rm (b)}] $J_1$ satisfies Freidlin's condition \cite{Freidlin_1985}:
\begin{equation}\label{eq:freidlin}
\begin{cases}
\medskip
\,\,J_1(t,x)= & \hspace{-0.3cm}\inf\limits_{\gamma(\cdot)\in\mathbb{X}^{t,x}}\Bigg\{\displaystyle\int_0^tL_1(t-s,\gamma(s),\dot\gamma(s))\,ds  \Bigg| (t-s,\gamma(s))\in P \text{ for } 0\leq s<t \Bigg\}\\
\,\,\text{for each }& \hspace{-0.3cm} (t,x)\in \partial P, \text { where } P:=\{(t,x):\,J_1(t,x)>0\};
\end{cases}
\end{equation}

\item[{\rm (c)}] There exists $\delta^*$ such that $J_1(t,x) = J_2(t,x)$ in $\{(t,x): x \geq (c_1 - \delta^*)t\}$.

\end{itemize}

\end{proposition}

\begin{lemma}\label{lem:3-4a}
Assume 
\eqref{eq:v_ep} holds for some $c_1 > \tilde{c}_1 \geq 2$. Then 
$$
w_* \geq w_2 \geq J_2 \quad \text{ in } (0,\infty)\times\mathbb{R},
$$
where $w_*$, $w_2$ and $J_2$ are given in \eqref{eq:w_star}, \eqref{eq:w1} and \eqref{eq:B3-1'}, respectively.
\end{lemma}
\begin{proof}
First, by adapting arguments in \cite[Lemma 3.1]{Evans_1989}, we show
\begin{equation}\label{goalwstar}
 w_*\geq w_2 \quad \text{ in } (0,\infty)\times  \mathbb{R}.
\end{equation}
Let $\eta>0$ and fix a function $\zeta\in C^\infty(\mathbb{R})$ satisfying
\begin{align}\label{eq:zeta}
\left\{\begin{array}{ll}
\medskip
\zeta\equiv 0 \text{ on } (-\infty,0], \zeta>0 \text{ on } (0,\infty),\\
0\leq \zeta\leq 1.
\end{array}
\right.
\end{align}

Consider now the auxiliary problem:
\begin{equation}\label{eq:problemeta}
\left\{
\begin{array}{ll}
\medskip
\min\{\partial_tw+H_2(t,x,\partial_x w),w\}=0,&\mathrm{in }~ (0,\infty)\times\mathbb{R},\\ 
w(0,x)=\eta\zeta,  &\text{on }  \mathbb{R}.
\end{array}
\right.
\end{equation}
Since the initial data of \eqref{eq:problemeta} is bounded, it follows from \cite[Theorem D.1]{Evans_1989} that \eqref{eq:problemeta} has a unique, Lipschitz  solution $w_{2,\eta}$ given by 
\begin{equation}\label{eq:w2mu}
\begin{split}
&w_{2,\eta}(t,x) =\\
& \sup_{\vartheta \in \Theta} \inf_{\gamma(\cdot) \in \mathbb{X}}\left\{ \int_0^{t \wedge \vartheta[\gamma(\cdot)]} \hspace{-0.1cm} L_2(t-s,\gamma(s),\dot\gamma(s))\,ds+  \chi_{\{\vartheta[\gamma(\cdot)] \geq t\}}\eta \zeta(\gamma(t))~
\Bigg|\gamma(0)=x\right\}.
\end{split}
\end{equation}
Since (i) $w_{2,\eta}(0,x)$ is uniformly bounded, (ii)  $w_*(0,x)\geq w_{2,\eta}(0,x)$ for all $x\in \mathbb{R}$ and (iii) $w_*$ is 
a viscosity super-solution of \eqref{eq:3-4}, it follows by comparison that 
$$w_*\geq w_{2,\eta} \quad \text{ in } (0,\infty)\times \mathbb{R}\text{ for each }\eta>0.$$
(Even though $w_*(0,x) = \infty$ for all $x>0$, it suffices to observe that $w_* - w_{2,\eta}$ cannot have negative interior minimum.
Here the fact that $w_{2,\eta}(0,x) <\infty$ for all $x$ is crucial, see \cite[Theorem B.1]{Evans_1989} for details). In what follows, we deduce $ w_{2,\eta}\to w_2$ as $\eta\to \infty$ and thus \eqref{goalwstar} holds.

Indeed, by \eqref{eq:w1} and \eqref{eq:w2mu}, it is easily seen that  $w_{2,\eta}$ is nondecreasing in $\eta$, and $w_{2,\eta}\leq w_2$ for all $\eta>0$, whence $w_{2,\eta}\to w_{2,\infty}$ pointwise  as $\eta\to\infty$ for some function  $w_{2,\infty}$ satisfying $0\leq w_{2,\infty}\leq w_{2}$. It remains to prove  $ w_{2,\infty}= w_{2}$. If not, then there are some $(t,x)\in(0,\infty)\times\mathbb{R}$, $\delta>0$ and $\eta_0>0$ such that
 \begin{equation}\label{conreadition}
   w_{2,\eta}(t,x)+3\delta< w_{2}(t,x)\,\quad \text{ for all } \eta\geq\eta_0.
 \end{equation}

According to definition \eqref{eq:w1}, we choose some $\vartheta_\infty\in \Theta$ such that
\begin{equation}\label{inequality10}
  w_2(t,x)\leq \inf_{\gamma(\cdot) \in \mathbb{X}^{t,x}} \int_0^{t \wedge \vartheta_\infty[\gamma(\cdot)]}L_2({t-s},\gamma(s),\dot\gamma(s))\,ds+\delta.
\end{equation}
By \eqref{eq:w2mu}, for any $\eta\geq\eta_0$ we further choose some $\gamma_\eta \in \mathbb{X}$ satisfying $\gamma_\eta(0)=x$  such that
\begin{equation}\label{inequality11}
  w_{2,\eta}(t,x) \geq   \int_0^{t \wedge \vartheta_\infty[\gamma_\eta(\cdot)]}L_2({t-s},\gamma_\eta(s),\dot\gamma_\eta(s))\,ds+  \chi_{\{\vartheta_\infty[\gamma_\eta(\cdot)] \geq t\}}\eta \zeta(\gamma_\eta(t))-\delta.
\end{equation}
Then we can reach a contradiction in two steps. First, we claim that $\vartheta_\infty[\gamma_{\eta}] \geq t$ for all $\eta \in [\eta_0,\infty)$. Suppose not, then
there exists some $\eta \in [\eta_0,\infty)$ such that  $\vartheta_\infty[\gamma_{\eta}]<t$.
 Then we can find some $\tilde{\gamma}_\eta\in\mathbb{X}$ such that
  $$\tilde{\gamma}_\eta=\gamma_\eta\quad \text{ in }\,\,[0,\vartheta_\infty[\gamma_\eta]]\,\quad \text{ and }\,\quad \tilde{\gamma}_\eta(t)= 0,$$
so that $\tilde{\gamma}_\eta\in\mathbb{X}^{t,x}$. Since $\vartheta_\infty[\gamma_{\eta}]<t$, by definition of stopping time, we get $\vartheta_\infty[\tilde{\gamma}_\eta]=\vartheta_\infty[\gamma_\eta]<t$.  Using \eqref{inequality10} and \eqref{inequality11}, we reach a contradiction:
       \begin{equation*}
    \begin{split}
  w_{2,\eta}(t,x)+2\delta\geq&
\int_0^{\vartheta_\infty[\tilde{\gamma}_\eta(\cdot)]}L_2({t-s},\tilde{\gamma}(s),\dot{\tilde{\gamma}}(s))\,ds+\delta\\
 \medskip
\geq&\inf_{\gamma(\cdot) \in \mathbb{X}^{t,x}} \int_0^{t \wedge \vartheta_\infty[\gamma(\cdot)]}L_2({t-s},\gamma(s),\dot\gamma(s))\,ds+\delta\geq w_2(t,x).
\end{split}
\end{equation*}

Hence, we must have $\vartheta_\infty[\gamma_{\eta}] \geq t$ for all $\eta \in [\eta_0,\infty)$, and \eqref{inequality11} becomes
\begin{equation}\label{inequality11b}
  w_{2,\eta}(t,x) \geq   \int_0^{t}L_2({t-s},\gamma_\eta(s),\dot\gamma_\eta(s))\,ds+  \eta \zeta(\gamma_\eta(t))-\delta.
\end{equation}

which implies the boundedness of $\{\gamma_\eta\}$ in $H^1([0,t])$. Indeed, by  \eqref{conreadition} and \eqref{inequality11b}
$$
\int_0^t \frac{|\dot{\gamma}_\eta|^2}{4}\,ds - 2t \leq \int_0^{t} L_2(t-s,\gamma_\eta, \dot{\gamma}_\eta)\,ds \leq w_{2,\eta}(t,x) + \delta \leq w_2(t,x) - 2\delta
$$
is independent of $\eta \geq \eta_0$. Then we obtain the boundedness of $\int_0^t |\gamma_\eta|^2\,ds$  by  
\begin{align*}
 \int_0^t |\gamma_\eta|^2ds&\leq 2\int_0^t |\gamma_\eta(s)-\gamma_\eta(0)|^2ds+2\int_0^t|\gamma_\eta(0)|^2ds\\
 &\leq  2\int_0^t \left[s \int_0^s|\dot{\gamma}_\eta(\hat{s})|^2d\hat{s}\right]ds+2x^2t\\
 &\leq 2Ct^2+2x^2t,
 \end{align*} 
 where we used $\gamma_\eta(0) =x$ and $\int_0^t |\dot{\gamma}_\eta|^2ds\leq C$ for some $C$ independent of $\eta$. Hence, $\{\gamma_\eta\}$ is uniformly bounded in $H^1([0,t])$, so that we may
 pass to a subsequence $\eta_n \to \infty$ so that $\gamma_{\eta_n}\rightharpoonup \gamma_\infty$ in $H^1([0,t])$ for some $\gamma_\infty \in \mathbb{X}$ satisfying $\gamma_\infty(0)=x$.
By \eqref{inequality11b}, we thus arrive at $\zeta(\gamma_\infty(t))=0$, so that $\gamma_\infty\in \mathbb{X}^{t,x}$  by \eqref{eq:zeta}.  Using \eqref{inequality10} and \eqref{inequality11b}, we have (using $t \wedge \vartheta_\infty[\gamma_{\eta}] = t$)
  $$\liminf_{n\to\infty} w_{2,\eta_n}(t,x)+2\delta\geq\inf_{\gamma(\cdot) \in \mathbb{X}^{t,x}} \int_0^{t}L_2({t-s},\gamma(s),\dot\gamma(s))\,ds+\delta\geq w_2(t,x),$$
 which contradicts \eqref{conreadition}.
Therefore, $w_{2,\infty} = w_2$ and \eqref{goalwstar} is proved.

Finally, the fact that $w_2 \geq J_2$ follows from definitions \eqref{eq:w1} and \eqref{eq:B3-1'} {by taking the stopping time $\vartheta \equiv \infty$ in \eqref{eq:w1}}.
\end{proof}

\begin{lemma}\label{lem:3-3}
Assume 
\eqref{eq:v_ep} holds for some $c_1 > \tilde{c}_1 \geq   2$.
 Then 
\begin{equation}\label{eq:lem331}
w^*(t,x) \leq w_1(t,x) \quad \text{ for }(t,x) \in (0,\infty)\times\mathbb{R},
\end{equation}
where $w^*$ and $w_1$ are given in \eqref{eq:w_star} and \eqref{eq:w1}, respectively. Furthermore, $w_1(t,x) = \max\{J_1(t,x),0\}$, where $J_1$ is defined in \eqref{eq:B3-1'}, so that
\begin{itemize}
\item[{\rm(a)}]
{
  If $\frac{c_1}{2}\in ( 1, \sqrt{a} + \sqrt{1-a}]$,
  then $c_1>\bar {c}_{\textup{nlp}}$ and
\begin{equation*}
\begin{split}
w_1(t,x) =
\left\{
\begin{array}{ll}
\medskip
\left(\frac{c_1}{2}-\sqrt{a}\right)(x-\bar{c}_{\textup{nlp}}t), & \text{for}~\bar{c}_{\textup{nlp}}<\frac{x}{t}\leq  c_1,\\
0 ,&  \text{for}~\frac{x}{t}\leq \bar{c}_{\textup{nlp}};
\end{array}
\right.
\end{split}
\end{equation*}}
\item[{\rm(b)}] If $\frac{c_1}{2} \in ( \sqrt{a} + \sqrt{1-a},\infty)$, 
then
\begin{equation*}
\begin{split}
w_1(t,x) =
\left\{
\begin{array}{ll}
\medskip
\left(\frac{c_1}{2}-\sqrt{a}\right)(x-\bar{c}_{\textup{nlp}}t), & \text{for}~c_1-2\sqrt{a}<\frac{x}{t}\leq  c_1,\\
\medskip
\frac{t}{4}(\frac{x^2}{t^2}-4(1-a)), &\text{for}~ 2\sqrt{1-a}< \frac{x}{t}\leq c_1-2\sqrt{a},\\
0, & \text{for}~\frac{x}{t}\leq 2\sqrt{1-a},
\end{array}
\right.
\end{split}
\end{equation*}
where $\bar{c}_{\textup{nlp}}=\frac{c_1}{2}-\sqrt{a}+\frac{1-a}{\frac{c_1}{2}-\sqrt{a}}$.
\end{itemize}

\end{lemma}
\begin{proof}
First,  we follow the strategy in \cite[Lemma 3.1]{Evans_1989} to show
\begin{equation}\label{eq:wls}
w^* \leq w_1 \quad \text{ in }  (0,\infty)\times\mathbb{R}.
\end{equation}

For each $\sigma\geq 0$, we  define $G_\sigma=(-\infty,-\sigma)$ and write
$$\Lambda_\rho\equiv \sup_{x\in G_\sigma}w^*(\rho,x)\quad\text{ for }\,\, \rho>0.$$
By Remark \ref{rmk:Qdelta}, $\Lambda_\rho <\infty$ for each $\rho >0$ and $w^*(0,x)=0$ for $x\in G_0$.
Since $w^*$ is upper semicontinuous, it follows that
\begin{equation}\label{Lrho}
\lim_{\rho\to 0}\Lambda_\rho=0 \quad \text{ for each }\sigma\geq 0.
\end{equation}
Choose some small $\rho>0$ and define function $w^{\sigma,\rho}_1: (\rho,\infty)\times\mathbb{R}$ by
\begin{equation}\label{eq:wsigmarho}
\begin{split}
w^{\sigma,\rho}_1(t,x) = \sup_{\vartheta \in \Theta} \inf_{\gamma(\cdot) \in \mathbb{X}} &\left\{ \int_0^{(t-\rho) \wedge \vartheta[\gamma(\cdot)]}L_1(t-s-\rho,\gamma(s),\dot\gamma(s))\,ds\right. \\
&\left.+ \Lambda_\rho \chi_{\{\vartheta[\gamma(\cdot)] \geq (t-\rho)\}} \Bigg|\gamma(0)=x \text{ and } \gamma(t-\rho)\in G_\sigma\right\}. 
\end{split}
\end{equation}
Then, by \cite[Theorem D.1]{Evans_1989}, $w^{\sigma,\rho}_1$ is a viscosity solution of
\begin{equation}\label{eq:3.7'}
\min\{\partial_t w+H_1(t,x,\partial_x w),w\}=0 \quad \mathrm{in}~ (\rho,\infty)\times\mathbb{R},
\end{equation}
 such that
$$w^{\sigma,\rho}_1(\rho,x) =g^{\sigma,\rho}(x):=\left\{\begin{array}{ll}\Lambda_\rho, &\text{ for }x\in  G_\sigma,\\
 \infty, &\text{ for } x\in \mathbb{R}\setminus G_\sigma.
\end{array}
\right.
$$

Note that $w^*(\rho,x) \leq w_1^{\sigma,\rho}(\rho,x)$ and $w^*(\rho,x) <\infty$ for all $x \in \mathbb{R}$, 
and  $w_1^{\sigma,\rho}$ and $w^*$ are respectively viscosity super- and sub-solutions of
\begin{equation*}
\min\{\partial_t w+H_1(t,x,\partial_x w),w\}=0 \quad  \mathrm{in}~ (\rho,\infty)\times\mathbb{R}.\end{equation*}
We may once again deduce by comparison \cite[Theorem B.1]{Evans_1989} that 
$$w^*\leq w_1^{\sigma,\rho} \quad\text{ in } (\rho,\infty)\times \mathbb{R}.$$
Let $\rho\to 0$  in \eqref{eq:wsigmarho} to discover $w^*\leq w_1^\sigma$,
where \begin{equation*}
\begin{split}
&w^{\sigma}_1(t,x)  \\
&= \sup_{\vartheta \in \Theta} \inf_{\gamma(\cdot) \in \mathbb{X}}\left\{ \int_0^{t \wedge \vartheta[\gamma(\cdot)]} L_1(t-s,\gamma(s),\dot\gamma(s))\,ds \,\Bigg|\,\gamma(0)=x \text{ and } \gamma(t)\in G_\sigma\right\}.
\end{split}
\end{equation*}
Letting $\sigma \to 0$ gives
$$
w_1^\sigma \to w_1 \quad\quad \text{ in } (0,\infty)\times \mathbb{R},
$$
and we arrive at \eqref{eq:wls}.

\indent It remains to show that $w_1 = \max\{J_1,0\}$. It follows from \eqref{eq:w1} that $w_1$ defines a locally Lipschitz viscosity solution of \eqref{eq:3-3}  (see \cite[Theorem 5.2]{Evans_1984} and  \cite[Theorem D.2]{Evans_1989}). Moreover, since $J_1$ verifies the Freidlin's condition \eqref{eq:freidlin} (see Proposition \ref{prop:B2}(b)),
we deduce $w_1(t,x) = \max\{J_1(t,x),0\}$ from \cite[Theorem 1]{Freidlin_1985} or \cite[Theorem 5.1]{Evans_1989}. 
This completes the proof of \eqref{eq:lem331}.

Finally, we verify $c_1 >\bar{c}_{\textup{nlp}}$, which implies that the ranges in the statement of the lemma are well-defined and lie within $P=\{(t,x):\, J_1(t,x)>0\}$.  Indeed, this follows from the direct calculation:
 $$c_1-\bar{c}_{\textup{nlp}}=\frac{c_1}{2}+\sqrt{a}-\frac{1-a}{\frac{c_1}{2}-\sqrt{a}}=\frac{c_1^2-4}{2c_1-4\sqrt{a}}>0,$$
as $c_1>2$. 
Hence, the formulas of $w_1$ follow from those of $J_1$ given in \eqref{eq:propA1}.
\end{proof}

\begin{lemma}\label{lem:3-4}
Assume 
\eqref{eq:v_ep} holds for some $c_1 > \tilde{c}_1 \geq 2$.
 Then there exists some $\delta^*>0$ such that
$$w^*=w_*=w_1=w_2\,\quad \text{ in  }\{(t,x): \,x\geq (c_1-\delta^*)t\},$$
where 
$w_1$ and $w_2$ are defined by \eqref{eq:w1}.
\end{lemma}

\begin{proof}
By definitions of $w^*$ and $w_*$ in \eqref{eq:w_star}, 
it is obvious that $w^*\geq w_*$. It remains to prove $w^*\leq w_*$ in $\{(t,x): \,x\geq (c_1-\delta^*)t\}$. By Lemmas \ref{lem:3-4a} and \ref{lem:3-3}, we have
$$
J_2 \leq w_2 \leq w_* \leq w^* \leq w_1 = \max\{J_1,0\}.
$$
By Proposition \ref{prop:B2}(c), there exists $\delta^*>0$ such that $J_1 = J_2>0$ in $\{(t,x): x \geq (c_1 - \delta^*) t\}$. This yields the desired conclusion.
\end{proof}

\begin{corollary}\label{cor3.9}
Let $\hat{c} = c_1 - \delta$ for some $\delta \in (0,\delta^*]$, then
$$
u(t, \hat{c}t) = \exp(-(\hat\mu + o(1))t)\,\quad \text{ for } t\gg1,
$$
where
$$
\hat\mu = \left(\frac{c_1}{2} - \sqrt{a}\right)(\hat{c} - \overline{c}_{\textup{nlp}}) \quad\text{ and }\quad   \overline{c}_{\textup{nlp}}= \frac{c_1}{2} - \sqrt{a} + \frac{1-a}{\frac{c_1}{2} - \sqrt{a}}.
$$
\end{corollary}
\begin{proof}
In view of $\delta \in (0,\delta^*]$, it follows from Lemma \ref{lem:3-4}
that for $0 < \epsilon \ll 1$,
\begin{equation*}
-\epsilon\log{u\left(\frac{1}{\epsilon},\frac{\hat{c}}{\epsilon}\right)}= \hat\mu + o(1) \,\, \Longleftrightarrow \,\, u\left(\frac{1}{\epsilon},\frac{\hat{c}}{\epsilon}\right) = \exp\left( -\frac{\hat\mu + o(1)}{\epsilon}\right)   \end{equation*}
where
\begin{equation*}
\hat\mu = w_1( 1, \hat{c}) = \left(\frac{c_1}{2}-\sqrt{a}\right)(\hat{c}-\bar{c}_{\textup{nlp}}), 
\end{equation*}
by Lemma \ref{lem:3-3}. The proof is complete.
\end{proof}

\section{Estimating $\underline{c}_2$ and $\overline{c}_2$}\label{sec:4}
In this section, we apply results in Section \ref{S3} with $c_1=2\sqrt{dr}$ and $\tilde{c}_1=2$
to determine the spreading speeds $\underline c_2$ and $\overline{c}_2$.

\begin{proposition}\label{thm:3.6}
Assume $\rm{(H_{\infty})}$ and $ dr>1$. Then $$\underline{c}_2 \geq {c}_{\textup{nlp}},$$ where ${c}_{\textup{nlp}}$ is given in \eqref{eq:c_acc-1} in the statement of Theorem \ref{thm:1-1}.

\end{proposition}
\begin{proof}
By Lemma \ref{lem:3-3},
\begin{equation}\label{eq:caccc}
  {\rm Int}\,\{(t,x):w_1(t,x) = 0\} = \{(t,x): x < {c}_{\textup{nlp}} t\}.
\end{equation}
We claim that it is enough to show that
\begin{equation}\label{eq:KK}
\liminf_{\epsilon \to 0} u^\epsilon(t,x) \geq {1-a} \quad \text{ uniformly on }K,
\end{equation}
for each compact subset $K$ of ${\rm Int}\,\{(t,x):w_1(t,x) = 0\}$. Granted, then for each $c < {c}_{\textup{nlp}}$, choose $K = \{(1,x): \, c-1 \leq x \leq c\}$,
so that
$$
K \subset  \{(t,x): x < {c}_{\textup{nlp}} t\}=  {\rm Int}\,\{(t,x):w_1(t,x) = 0\}.
$$
Then
$$
\liminf_{t \to \infty} \inf_{ct-1 < x < ct} u(t,x) = \liminf_{\epsilon \to 0} \inf_{K } u^\epsilon  \geq 1-a >0,
$$
i.e., $\underline{c}_2\geq c$ for all $c<{c}_{\textup{nlp}}$, so that $\underline{c}_2\geq c_{\textup{nlp}}$.

To prove \eqref{eq:KK}, we recall the arguments in \cite[Section 4]{Evans_1989}. Let $K$ and $K'$ be compact subsets so that $K\subset {\rm Int}\, K' \subset K' \subset {\rm Int}\,\{(t,x):w_1(t,x) = 0\}$. By \eqref{eq:lem331} in Lemma \ref{lem:3-3} and {$0\leq w_*\leq  w^*\leq w_1$}, we have $w_*(t,x) =w^*(t,x) = 0$ in {$\{(t,x):w_1(t,x) = 0\}$.}
Hence, we have $w^\epsilon(t,x) \to 0$ uniformly in $K'$. Fix $(t_0,x_0) \in K$ and consider the test function
$$
\rho(t,x) = |x-x_0|^2 + (t-t_0)^2.
$$
Then for all small $\epsilon$, the function $w^\epsilon - \rho$ has a local maximum point $(t_\epsilon,x_\epsilon)$ such that  $(t_\epsilon,x_\epsilon)\to (t_0,x_0)$ as $\epsilon \to 0$. Furthermore, $\partial_t \rho(t_\epsilon,x_\epsilon), \partial_x \rho(t_\epsilon,x_\epsilon) \to 0$, so that at the point $(t_\epsilon,x_\epsilon)$,
$$
o(1) = \partial_t \rho - \epsilon \partial_{xx} \rho + |\partial_x\rho|^2 \leq \partial_t w^\epsilon - \epsilon \partial_{xx} w^\epsilon + |\partial_x w^\epsilon|^2 \leq u^\epsilon - 1 + a,
$$
where the second inequality is due to $v^\epsilon \leq 1$. This yields
$$
u^\epsilon(t_\epsilon, x_\epsilon) \geq 1-a + o(1).
$$
In light of $w^\epsilon(t_\epsilon,x_\epsilon)  \geq (w^\epsilon - \rho)(t_\epsilon,x_\epsilon) \geq(w^\epsilon - \rho)(t_0,x_0) = w^\epsilon(t_0,x_0)$, we have
$$
-\epsilon \log u^\epsilon(t_\epsilon,x_\epsilon)= w^\epsilon(t_\epsilon,x_\epsilon)\geq w^\epsilon(t_0,x_0) =-\epsilon \log u^\epsilon(t_0,x_0),
$$
so that $u^\epsilon(t_0,x_0) \geq u^\epsilon(t_\epsilon,x_\epsilon) \geq 1-a + o(1).$ Since this argument is uniform for $(t_0,x_0) \in K$, we deduce \eqref{eq:KK}.  This completes the proof.
\end{proof}

\begin{proposition}\label{thm:3-7}
Under the assumption $\rm{(H_{\infty})}$ and $dr>1$, we have
$$\overline{c}_2 \leq  \max\{{c}_{\textup{LLW}},{c}_{\textup{nlp}}\},$$
where ${c}_{\textup{LLW}}$  is given by Theorem \ref{thm:LLW}  and  ${c}_{\textup{nlp}}$ is defined by \eqref{eq:c_acc-1}.
\end{proposition}
\begin{proof}
Denote $\hat{c} = c_1 - \delta^*$, where $c_1 = 2\sqrt{dr}$ and $\delta^*$ is given by Lemma \ref{lem:3-4}.
It follows from Corollary \ref{cor3.9} 
that
$$
u(t, \hat{c}t) = \exp(-(\hat\mu + o(1))t)\,\quad \text{ for } t\gg1.
$$
Here
\begin{equation}\label{eqeq:mu}
\hat\mu =  \left(\frac{c_1}{2}-\sqrt{a}\right)(\hat{c}-\bar{c}_{\textup{nlp}})  =  \hat{c}\left(\frac{c_1}{2}-\sqrt{a}\right) - \left(\frac{c_1}{2}-\sqrt{a}\right)^2 - (1-a), 
\end{equation}
where we used
\begin{equation}\label{eq:barcnlp}
\bar{c}_{\textup{nlp}}=\frac{c_1}{2}-\sqrt{a}+\frac{1-a}{\frac{c_1}{2}-\sqrt{a}}.
\end{equation}
 We note for later purposes that \eqref{eqeq:mu} and \eqref{eq:barcnlp} are quadratic equations in $\frac{c_1}{2} - \sqrt{a}$, so that
\begin{equation}\label{eq:later}
\frac{c_1}{2} - \sqrt{a} = \frac{\hat{c} - \sqrt{\hat{c}^2 - 4(\hat\mu + 1 - a)}}{2}= \frac{\bar{c}_{\textup{nlp}} - \sqrt{\bar{c}_{\textup{nlp}}^2 - 4(1 - a)}}{2}.
\end{equation}

 Moreover, by Proposition \ref{prop:1}, we arrive at
  $$\lim\limits_{t\to\infty} (u,v)(t,\hat{c} t)=(0,1)\text{ and }\lim\limits_{t\to\infty}(u,v)(t,0)=(k_1,k_2).$$
  We may then apply Lemma \ref{lem:appen1}{\rm{(a)}}  to conclude
\begin{equation}\label{eq:nu1}
\overline{c}_2 \leq  c_{\hat{c},\hat\mu}=\left\{
\begin{array}{ll}
\medskip
{c}_{\textup{LLW}},& \text{ if } \hat\mu\geq \lambda_{\textup{LLW}}(\hat c- {c}_{\textup{LLW}}),\\
\hat c-\frac{2\hat\mu}{\hat c-\sqrt{\hat c^2-4(\hat\mu+1-a)}},&  \text{ if } \hat\mu<\lambda_{\textup{LLW}}(\hat c- {c}_{\textup{LLW}}).\\
\end{array}
\right .
\end{equation}
To complete the proof, we just need to verify $ c_{\hat{c},\hat\mu}= \max\{{c}_{\textup{LLW}}, {c}_{\textup{nlp}}\}$, where
\begin{equation*}
\begin{array}{ll}
{c}_{\textup{nlp}}= \left\{
\begin{array}{ll}
\medskip
\bar {c}_{\textup{nlp}},  & \qquad  \qquad \text{ if }\,\,c_1 \leq 2(\sqrt a +\sqrt{1-a}),\\
2\sqrt{1-a}, & \qquad \qquad \text{ if }\,\,c_1 > (\sqrt a +\sqrt{1-a}).
\end{array}
\right.
\end{array}
\end{equation*}
\begin{itemize}

\item[(i)] For the case $\frac{c_1}{2} - \sqrt{a} < \lambda_{\rm LLW}$, we have $\frac{c_1}{2} - \sqrt{a} < \lambda_{\rm LLW} \leq \sqrt{1-a}$, so that
$$
c_{\rm nlp} =\bar{c}_{\rm nlp} = \frac{c_1}{2} - \sqrt{a} + \frac{1-a}{\frac{c_1}{2} - \sqrt{a}} > \lambda_{\rm LLW} + \frac{1-a}{\lambda_{\rm LLW}} = c_{\rm LLW}.
$$
(Note that $h(s) = s + \frac{1-a}{s}$ is strictly decreasing in $(0,\sqrt{1-a}]$.)
  Hence,
   \begin{equation*}
\begin{array}{l}
\hat\mu=(\frac{c_1}{2}-\sqrt{a})(\hat c-\bar{c}_{\textup{nlp}})  <\lambda_{\textup{LLW}}(\hat c- {c}_{\textup{LLW}})
\end{array}
\end{equation*}
  so that, by \eqref{eq:nu1}, $ c_{\hat{c},\hat\mu}=\hat c- \frac{2\hat\mu}{\hat c-\sqrt{\hat c^2-4(\hat\mu+1-a)}} $. Using
  \eqref{eqeq:mu} and  \eqref{eq:later},
$$
 c_{\hat{c},\hat\mu}= \hat{c} -   \frac{1}{ \frac{c_1}{2} - \sqrt{a}}\left( \frac{c_1}{2} - \sqrt{a}\right)(\hat{c} - \bar{c}_{\textup{nlp}})=\bar{c}_{\textup{nlp}} = \max\{c_{\rm nlp}, c_{\rm LLW}\};
$$
\item[(ii)] For the case $\frac{c_1}{2} - \sqrt{a} \geq \lambda_{\rm LLW}$, we have $c_{\rm nlp} \leq  c_{\rm LLW}$.
By \eqref{eqeq:mu} and the  fact that $c_{\rm LLW} =\lambda_{\rm LLW} + \frac{1-a}{\lambda_{\rm LLW}}$, we derive that
$$\hat\mu = \left( \frac{c_1}{2} - \sqrt{a}\right) \hat{c} -\left( \frac{c_1}{2} - \sqrt{a}\right) ^2-(1-a)\geq\lambda_{\textup{LLW}}\hat c-\lambda_{\textup{LLW}}^2-(1-a)=\lambda_{\textup{LLW}}(\hat c- c_{\textup{LLW}}),$$
where the inequality holds since $\lambda\hat c-\lambda^2-(1-a)$ is an increasing function of $\lambda$
in $(0,\frac{\hat c}{2})$. Thus by  \eqref{eq:nu1}, $ c_{\hat{c},\hat\mu}= {c}_{\textup{LLW}} = \max\{c_{\rm nlp}, c_{\rm LLW}\}$ as desired.
\end{itemize}
This completes the proof of Proposition \ref{thm:3-7}.
\end{proof}

\begin{proof}[Proof of Theorem \ref{thm:1-1}]
Let ${c}_{\textup{nlp}}$ be as given in \eqref{eq:c_acc-1} in the statement of Theorem \ref{thm:1-1}.
By Proposition \ref{prop:1}, it remains to show that $\underline{c}_2 \geq c_{\textup{nlp}}$ and $\overline{c}_2 \leq \max\{{c}_{\textup{LLW}},{c}_{\textup{nlp}}\}$. These are proved in Propositions \ref{thm:3.6} and \ref{thm:3-7} respectively.
\end{proof}

\section{The case $dr=1$ \label{sec:5}}
Here, 
we prove Theorem \ref{thm:1-1'} by applying Theorem \ref{thm:1-1}.
\begin{proof}[Proof of Theorem \ref{thm:1-1'}] 
Let $(u,v)$ be a solution of \eqref{eq:1-1} with initial data $(u_0,v_0)$ satisfying $\rm(H_\infty)$.
For any small $\delta\in(0,1)$, let $(\underline{u}^\delta,\overline{v}^\delta)$ and  $(\overline{u}^\delta,\underline{v}^\delta)$ be respectively the solutions of
\begin{equation}\label{eq:1-1'}
\left\{
\begin{array}{ll}
\medskip
\partial_t u-\partial_{xx}u=u(1-u-av),& \text{ in }(0,\infty)\times \mathbb{R},\\
\partial _t v-d\partial_{xx}v=rv(1+\delta-bu-v),& \text{ in }(0,\infty)\times \mathbb{R},\\
\end{array}
\right .
\end{equation}
and
\begin{equation}\label{eq:1-1''}
\left\{
\begin{array}{ll}
\medskip
\partial_t u-\partial_{xx}u=u(1-u-av),& \text{ in }(0,\infty)\times \mathbb{R},\\
\partial _t v-d\partial_{xx}v=r v(1-\delta-bu-v),& \text{ in }(0,\infty)\times \mathbb{R},\\
\end{array}
\right .
\end{equation}
with the same initial data $(u_0,v_0)$. By comparison, we deduce that
\begin{equation}\label{eq:uvsupersub}
(\underline{u}^\delta,\overline{v}^\delta)\preceq (u,v)\preceq (\overline{u}^\delta,\underline{v}^\delta) \,\, \text{ in } [0,\infty)\times\mathbb{R}.
\end{equation}

Notice that $(\underline{u}^\delta,\overline{v}^\delta)$ is a solution of \eqref{eq:1-1'} if and only if
\begin{equation}\label{scaling1}
  (\underline{U}^\delta,\overline V^\delta)= \left(\underline{u},\frac{\overline{v}^\delta}{1+\delta}\right)
\end{equation}
 is a solution of
 \begin{equation}\label{eq:1-1'trs}
\left\{
\begin{array}{ll}
\medskip
\partial_t U-\partial_{xx}U=U(1-U-\overline{a}^\delta V),& \text{ in }(0,\infty)\times \mathbb{R},\\
\partial _t V-d\partial_{xx}V=\overline{r}^\delta V(1-\underline{b}^\delta U-V),& \text{ in }(0,\infty)\times \mathbb{R},\\
\end{array}
\right .
\end{equation}
where $\overline{a}^\delta=(1+\delta)a,\, \overline{r}^\delta=(1+\delta)r$ and $ \underline{b}^\delta=\frac{b}{1+\delta}$.
Observe that $d\overline{r}^\delta>1$ and  $0<\overline{a}^\delta,\underline{b}^\delta<1$ by choosing $\delta$ small enough. By applying Theorem \ref{thm:1-1} to \eqref{eq:1-1'trs} and using \eqref{scaling1}, we deduce that
for each small $\eta>0$,
 \begin{equation}\label{spreadingdelta1}
\begin{cases}
\lim\limits_{t\rightarrow \infty} \sup\limits_{ x>(\overline{c}_{1}^\delta+\eta) t} (|\underline{u}^\delta(t,x)|+|\overline{v}^\delta(t,x)|)=0, \\
\lim\limits_{t\rightarrow \infty} \sup\limits_{(\underline{c}_2^\delta+\eta) t< x<(\overline{c}_{1}^\delta-\eta) t} (|\underline{u}^\delta(t,x)|+|\overline{v}^\delta(t,x)-(1+\delta)|)=0, \\
\lim\limits_{t\rightarrow \infty} \sup\limits_{(\underline{c}_{3}^\delta+\eta)t< x<(\underline{c}_2^\delta-\eta) t} (|\underline{u}^\delta(t,x)-\underline{k}_1^\delta|+|\overline{v}^\delta(t,x)-(1+\delta)\overline{k}_2^\delta|)=0 , \\
 \lim\limits_{t\rightarrow \infty} \sup\limits_{x<(\underline{c}_{3}^\delta-\eta)t} (|\underline{u}^\delta(t,x)-1|+|\overline{v}^\delta(t,x)|)=0, 
 \end{cases}
\end{equation}
where
$$\overline{c}_1^\delta = 2\sqrt{d\overline{r}^\delta}, \quad \underline{c}_2^\delta = \max\{\underline{c}_{\textup{LLW}}^\delta, \underline{c}_{\textup{nlp}}^\delta\},\quad \underline{c}_3^\delta = -\overline{\tilde{c}}_{\textup{LLW}}^\delta,
$$
$$(\underline{k}_1^\delta,\overline{k}_2^\delta)=\left(\frac{1-\overline{a}^\delta}{1-\overline{a}^\delta\underline{b}^\delta},\frac{1-\underline{b}^\delta}{1-\overline{a}^\delta \underline{b}^\delta}\right)\rightarrow(k_1,k_2)\,\,\text{ as } \delta\rightarrow0,$$
and $\underline{c}_{\textup{LLW}}^\delta$
$(\text{resp.}~\overline{\tilde{c}}_{\textup{LLW}}^\delta)$ is the spreading speed for \eqref{eq:1-1'trs}  as given in Theorem \ref{thm:LLW} $($resp. Remark \ref{rmk:LLW}$)$, and
\begin{equation*}
\begin{array}{ll}
\underline{c}_{\textup{nlp}}^\delta= \left\{
\begin{array}{ll}
\sqrt{d\overline{r}^\delta} - \sqrt{\overline{a}^\delta} + \frac{1-\overline{a}^\delta}{\sqrt{d\overline{r}^\delta} - \sqrt{\overline{a}^\delta}},  & \qquad  \qquad \text{ if }\,\,\sqrt{d\overline{r}^\delta} \leq \sqrt{\overline{a}^\delta} +\sqrt{1-\overline{a}^\delta},\\
2\sqrt{1-\overline{a}^\delta}, & \qquad \qquad \text{ if }\,\,\sqrt{d\overline{r}^\delta} > \sqrt{\overline{a}^\delta} +\sqrt{1-\overline{a}^\delta}.
\end{array}
\right.
\end{array}
\end{equation*}
Together with \eqref{eq:uvsupersub}, \eqref{spreadingdelta1} implies particularly that
$$
\overline{c}_1\leq \overline{c}_{1}^\delta,\quad \underline{c}_2\geq \underline{c}_2^\delta\quad \text{ and } \quad \underline{c}_3\geq \underline{c}_{3}^\delta.
$$
By the continuity of  $\underline{c}_{\textup{LLW}}^\delta$, $\overline{\tilde{c}}_{\textup{LLW}}^\delta$ in $\delta$ (see, e.g., \cite[Theorem 4.2 of Ch. 3]{Volpert_1994}), letting $\delta\to 0$ yields
\begin{equation}\label{eq:upperci}
\overline{c}_1\leq 2,\quad \underline{c}_2\geq 2\quad \text{ and }\quad \underline{c}_3\geq -\tilde{c}_{\textup{LLW}}.
\end{equation}

Similarly, by observing that
$(\overline{u}^\delta,\underline{v}^\delta)$ is a solution of \eqref{eq:1-1''} if and only if
 $$(\overline{U}^\delta,\underline V^\delta)=\left(\overline{u}^\delta,\frac{\underline{v}^\delta}{1-\delta}\right)$$
 is a solution of
 \begin{equation}\label{eq:1-1''trs}
\left\{
\begin{array}{ll}
\medskip
\partial_t U-\partial_{xx}U=U(1-U-\underline{a}^\delta V),& \text{ in }(0,\infty)\times \mathbb{R},\\
\partial _t V-d\partial_{xx}V=\underline{r}^\delta V(1-\overline{b}^\delta U-V),& \text{ in }(0,\infty)\times \mathbb{R},\\
\end{array}
\right .
\end{equation}
where $\underline{a}^\delta=(1-\delta)a$, $ \underline{r}^\delta=(1-\delta)r$ and $\overline{b}^\delta=\frac{b}{1-\delta}$.
This time, $d\underline{r}^\delta<1$ and  $0<\underline{a}^\delta,\overline{b}^\delta<1$ by choosing $\delta$ small enough. We apply Corollary \ref{rmk:drleq1} to \eqref{eq:1-1''trs}. In view of \eqref{eq:uvsupersub} and letting $\delta\to0$, we deduce
\begin{equation}\label{eq:lowerci}
\underline{c}_1\geq 2, \quad \overline{c}_2\leq 2 \quad\text{ and}  \quad \overline{c}_3\leq -\tilde{c}_{\textup{LLW}}.
\end{equation}
By definition of $\underline{c}_i$ and $\overline{c}_i$, we deduce $\underline{c}_i\leq \overline{c}_i$ for $i=1,\,2,\,3$. With \eqref{eq:upperci} and \eqref{eq:lowerci}, we obtain Theorem \ref{thm:1-1'}.
\end{proof}

\section*{Acknowledgement}
The authors wish to thank the two referees for his/her suggestions which have improved the paper and led to the addition of Theorem \ref{thm:1-1'}.
QL (201706360310) and SL (201806360223) would like to thank the China Scholarship Council for financial support during the period of their overseas study and  express their gratitude to the Department of Mathematics, The Ohio State University for the warm hospitality. SL is partially supported by the Outstanding Innovative Talents Cultivation Funded Programs 2018 of Renmin University of China. 

\appendixtitleon
\appendixtitletocon
\begin{appendices}

\section{Proof of Lemma \ref{lem:appen1}}\label{sec:B}
In this section, we prove Lemma \ref{lem:appen1}, which was used in proving Proposition \ref{prop:1} and Theorem \ref{thm:3-7}.

\begin{proof}[Proof of Lemma \ref{lem:appen1}]
We only prove {\rm(a)}, as  {\rm(b)} can be proved by similar arguments.

\noindent {\bf Step 1. } We first show
\begin{equation}\label{sup_uv}
 \limsup_{t \to \infty} \sup_{0 \leq x \leq \hat{c}t}\tilde u(t,x) \leq k_1, \quad   \liminf_{t \to \infty} \inf_{0 \leq x \leq \hat{c}t}\tilde v(t,x) \geq k_2.
\end{equation}
By Lemma \ref{lem:entire1}{\rm(a)}, it suffices to show $\displaystyle \liminf_{t \to \infty} \inf_{0 \leq x \leq \hat{c}t} \tilde v(t,x) >0$. Since
 $$\lim\limits_{t\to\infty}(\tilde u,\tilde v)(t,0)=(k_1,k_2) \text{ and } \lim\limits_{t\to\infty}(\tilde u,\tilde v)(t,\hat ct)=(0,1),$$
we can fix $T>0$ such that
$$
\inf_{t \geq T}\{\tilde v(t, 0), \tilde v(t, \hat ct)\} >0.
$$
Define
 $
\delta':= \min\left\{\frac{1-b}{2},\, \inf\limits_{0 < x < \hat cT}\tilde v(T,x),\,\inf\limits_{t \geq T}\{\tilde v(t, 0), \tilde v(t, \hat ct)\} \right\} >0.
$
Note that $\tilde v$ is a super-solution to the KPP-type equation
$\partial_t \tilde v = d\partial_{xx}\tilde v + r\tilde v(1-b -\tilde  v)$ in the domain $\Omega':=\{(t,x): t \geq T,\,  0\leq x \leq \hat c t\}$
such that $\tilde v(t,x) \geq \delta' >0$ on the parabolic boundary. Since $\tilde{v} - \delta'$ cannot attain negative interior minimum,
 we deduce that  $\tilde v(t,x) \geq \delta'$ in $\Omega'$, which completes Step 1.

For a small $\delta>0$ to be determined later, consider
\begin{align}\label{eq:4-1}
\left\{
\begin{array}{ll}
\medskip
\partial_t \tilde u=\partial_{xx}\tilde u+\tilde u(1+2\delta-\tilde u-a \tilde v) & \mathrm{in} ~(0,\infty)\times \mathbb{R},\\
\partial_t \tilde v=d \partial_{xx}\tilde v+r \tilde v(1-2\delta-b \tilde u-\tilde v) & \mathrm{in}~(0,\infty)\times \mathbb{R}.\\
\end{array}
\right.
\end{align}
Denote by ${c}_{\textup{LLW}}^\delta$  the spreading speed for the homogeneous coexistence equilibrium $(k_1^{\delta},k_2^{\delta})$ of \eqref{eq:4-1}  into the region where $(\tilde u, \tilde v) \approx (0,1-2\delta)$. By continuous dependence on parameters , 
 $c^\delta_{\textup{LLW}} \to c_{\textup{LLW}}$ as $\delta\to0$ \cite[Theorem 4.2 of Ch. 3]{Volpert_1994}, where $c_{\rm LLW}$ is given in Theorem \ref{thm:LLW}. We now define
\begin{equation}\label{eq:lambdahat}
\lambda_{\hat c,\hat\mu}=\frac{1}{2}\left(\hat c-\sqrt{\hat c^2-4(\hat\mu+1-a)}\right),
\end{equation}
which satisfies $-\lambda^2_{\hat c,\hat\mu}+\lambda_{\hat c,\hat\mu}\hat c-(1-a)=\hat\mu$. In view of definition of $c_{\hat c,\hat\mu}$ in the statement of Lemma \ref{lem:appen1}, $\hat\mu$ and $c_{\hat c,\hat\mu}$ can be rewritten as
\begin{equation}\label{eq:crewrite}
\hat\mu = \lambda_{\hat c,\hat\mu}(\hat c- {c}_{\hat{c},\hat\mu}),
\quad
{c}_{\hat{c},\hat\mu}=\left\{
\begin{array}{ll}
\medskip
{c}_{\textup{LLW}},& \text{ if } \hat\mu\geq \lambda_{\textup{LLW}}(\hat{c} - {c}_{\textup{LLW}}),\\
\lambda_{\hat c,\hat\mu}+\frac{1-a}{\lambda_{\hat c,\hat\mu}},&  \text{ if } \hat\mu< \lambda_{\textup{LLW}}(\hat{c} - {c}_{\textup{LLW}}).\\
\end{array}
\right. \end{equation}

\noindent{\bf Step 2.} Assume $ \hat\mu< \lambda_{\rm LLW}(\hat{c} - {c}_{\textup{LLW}})$, so that by \eqref{eq:crewrite} we have
\begin{equation}\label{eq:B4b}
{c}_{\hat{c},\hat\mu}= \lambda_{\hat c,\hat\mu}+\frac{1-a}{\lambda_{\hat c,\hat\mu}} \quad \Longleftrightarrow \quad \lambda_{\hat c,\hat\mu}^2 - \lambda_{\hat c,\hat\mu}c_{\hat c,\hat\mu} + (1-a) = 0. 
\end{equation}
We show that
\begin{equation}\label{sup_c}
\begin{array}{l}
\lim\limits_{t\rightarrow \infty}\sup\limits_{ x>ct}\tilde u(t,x)=0 \,\,\text{ for any }  c>{c}_{\hat{c},\hat\mu}.
\end{array}
\end{equation}

First, we claim ${c}_{\hat{c},\hat\mu}> {c}_{\textup{LLW}}$.  Considering the auxiliary function $$f(z)=\frac{z-\sqrt{z^2-4(1-a)}}{2}(\hat c-z).$$
By direct calculation, $f$ is decreasing in $[2\sqrt{1-a},\hat{c}]$. In view of \eqref{eq:LLL} and \eqref{eq:B4b}, we have $f({c}_{\hat{c},\hat\mu})=\hat\mu$ and $f( {c}_{\textup{LLW}})=\lambda_{\textup{LLW}}(\hat{c} - {c}_{\textup{LLW}})$. Since $ \hat\mu< \lambda_{\textup{LLW}}(\hat{c} - {c}_{\textup{LLW}})$ and $f$ is decreasing, we deduce ${c}_{\hat{c},\hat\mu}> {c}_{\textup{LLW}}$.

 Let $\lambda_{\hat c,\hat\mu}^\delta=\lambda_{\hat c,\hat\mu}-\delta$ and $c_{\hat c,\hat\mu}^\delta=\lambda_{\hat c,\hat\mu}^\delta+\frac{1-a+2\delta(1+a)}{\lambda_{\hat c,\hat\mu}^\delta}$. In view of
\begin{equation}\label{eq:B5b}
\begin{array}{ll}
(\lambda_{\hat c,\hat\mu}^\delta)^2-\lambda_{\hat c,\hat\mu}^\delta {c}_{\hat{c},\hat\mu}^\delta+(1-a)+2\delta(1+a)=0,
\end{array}
\end{equation}
and \eqref{eq:crewrite}, we obtain the following inequality which will be useful later.
\begin{equation}\label{eq:4.00}
    \begin{split}
    \hat\mu-\lambda_{\hat c,\hat\mu}^\delta (\hat c-{c}_{\hat{c},\hat\mu}^\delta)&=\lambda_{\hat c,\hat\mu}(\hat c- {c}_{\hat{c},\hat\mu})-\lambda_{\hat c,\hat\mu}^\delta (\hat c- {c}_{\hat{c},\hat\mu}^\delta)\\
   &=\hat c\delta  -\lambda_{\hat c,\hat\mu} {c}_{\hat{c},\hat\mu}+\lambda_{\hat c,\hat\mu}^\delta {c}_{\hat{c},\hat\mu}^\delta\\
   &>\hat c\delta-\lambda_{\hat c,\hat\mu}^2+ (\lambda_{\hat c,\hat\mu}^\delta)^2\\
   &=\delta(\sqrt{ \hat c^2-4(\hat\mu+1-a)}+\delta)>0,
\end{split}
\end{equation}
where we used \eqref{eq:B4b} and \eqref{eq:B5b} for the inequality, and used  $\lambda_{\hat c,\hat\mu}^\delta=\lambda_{\hat c,\hat\mu}-\delta$ and  \eqref{eq:lambdahat} for the last equality.

Since ${c}_{\hat{c},\hat\mu}>{c}_{\textup{LLW}}$, by the continuity of $ {c}_{\hat{c},\hat\mu}^\delta$ and ${c}_{\textup{LLW}}^\delta$ in $\delta$ (see, e.g., \cite[Theorem 4.2 of Ch. 3]{Volpert_1994}), we select $\delta$ so small that  ${c}_{\hat{c},\hat\mu}^\delta>{c}_{\textup{LLW}}^\delta$. Since $c^\delta_{\rm LLW}$ is the minimal traveling wave speed, this ensures the existence of the traveling wave solution  with speed ${c}_{\hat{c},\hat\mu}^\delta$ for \eqref{eq:4-1}. Let $(\varphi^\delta,\psi^\delta)(s)$ be such a traveling wave solution normalized by $\varphi^\delta(0)=k_1+\delta$ satisfying
\begin{equation}\label{eq:4-2}
\left\{
\begin{array}{l}
\smallskip
-{c}_{\hat{c},\hat\mu}^\delta\varphi'=\varphi''+\varphi(1+2\delta-\varphi-a\psi),~~~~~ s\in \mathbb{R},\\
\smallskip
-{c}_{\hat{c},\hat\mu}^\delta\psi'=d\psi''+r\psi(1-2\delta-b\varphi-\psi) ,~~ s\in \mathbb{R},\\
(\varphi, \psi)(-\infty)=(k_1^{\delta},k_2^{\delta}), ~(\varphi, \psi)(\infty)=(0,1-2\delta).
\end{array}
\right.
\end{equation}

To establish \eqref{sup_c}, we first prove that there exist $T_1$ and $x_1$ such that
 \begin{equation}\label{eq:4-3}
    ( \tilde u, \tilde v)(t,x) \preceq (\varphi^\delta, \psi^\delta)(x- c_{\hat c,\hat\mu}^\delta t-x_1) ~\mathrm{for}~t\geq T_1,~0\leq x\leq \hat ct.
\end{equation}
To apply the comparison principle, we need to verify the following conditions:
\begin{itemize}
\item[(i)]{ $(\tilde u,\tilde v)(T_1,x)\preceq (\varphi^\delta,\psi^\delta)(x-{c}_{\hat{c},\hat\mu}^\delta T_1-x_1)$     for $0\leq x\leq \hat cT_1$;}
\item[(ii)]{ $(\tilde u,\tilde v)(t,0)\preceq (\varphi^\delta,\psi^\delta) (-{c}_{\hat{c},\hat\mu}^\delta t-x_1)$     for $t\geq T_1$;}
\item[(iii)]{ $(\tilde u,\tilde v)(t,\hat ct)\preceq (\varphi^\delta,\psi^\delta) ((\hat c-{c}_{\hat{c},\hat\mu}^\delta)t-x_1)$     for $t\geq T_1$.}
\end{itemize}

First, we verify condition (iii). Since $\lim\limits_{t\to\infty}\tilde v(t,\hat ct)=1$, we choose $T_2$ such that
$$\tilde v(t,\hat ct)\geq 1-\delta\geq \psi^\delta(\infty)\geq \psi^\delta((\hat c-{c}_{\hat{c},\hat\mu}^\delta)t-x_1) \text{ for all } t\geq T_2 \, \text{ and }\, x_1 \geq 0.$$
Also, since $c_{\hat c,\hat\mu}^\delta>c^\delta_{\rm LLW}$, the expression of $\varphi^\delta$ at infinity {(see, e.g., \cite{Hou_2008})} can be described by
$$\varphi^\delta(s)=A\exp\left\{-(\lambda_{\hat c, \hat\mu}^\delta+o(1)) s\right\}~~~~as~~s \rightarrow \infty.$$
Recalling \eqref{eq:4.00}, we have
$ \hat\mu>\lambda_{\hat c,\hat\mu}^\delta (\hat c-{c}_{\hat{c},\hat\mu}^\delta)$.  Noting that, {by hypothesis of the lemma,} $\tilde u(t,\hat ct)\leq \exp\{-(\hat\mu+o(1)) t\}$ as  $t\to\infty$.
We can  choose $T_1>T_2$ such that $$\tilde u(t,\hat ct)<\varphi^\delta((\hat c-{c}_{\hat{c},\hat\mu}^\delta)t-x_1) \quad\text{ for } t\geq T_1 \, \text{ and }\, x_1 \geq 0,$$
which verifies (iii). Next, we choose (by Step 1)  $x_1 \gg 1$ so that {\rm(i)} and {\rm(ii)} hold. This allows the application of the comparison principle to establish  \eqref{eq:4-3}.

Therefore, for each $\delta>0$, we arrive at
\begin{equation*}
\begin{array}{l}
\limsup\limits_{t\rightarrow \infty}\sup\limits_{ x> ct}\tilde u(t,x)\leq \limsup\limits_{t\rightarrow \infty}\sup\limits_{ x> ct} \varphi^\delta(x- {c}_{\hat{c},\hat\mu}^\delta t-x_1)=0 \quad \text{ for }c>{c}_{\hat{c},\hat\mu}^\delta.
\end{array}
\end{equation*}
Since the above is true for all $\delta>0$, we deduce that
$$\limsup\limits_{t\rightarrow \infty}\sup\limits_{ x> ct}\tilde u(t,x)\leq 0, \quad  \text{ for each }\,\, c > \lim_{\delta\to 0}{c}_{\hat{c},\hat\mu}^\delta = {c}_{\hat{c},\hat\mu}.$$ Thus \eqref{sup_c} holds.

\noindent{\bf Step 3.} Assume $\hat\mu \geq \lambda_{\textup{LLW}}( \hat{c} - {c}_{\textup{LLW}})$. Then, for each $0<\hat\mu' < \lambda_{\textup{LLW}}( \hat{c} - {c}_{\textup{LLW}})$, we have
$$
\tilde u(t, \hat{c} t) \leq \exp(-\hat\mu' t)   \quad \text{ for all }t \gg 1.
$$
Hence, we may repeat Step 2 to deduce that 
\begin{equation}\label{eq:theabove}
\lim_{t\to\infty} \sup_{x > ct} \tilde u(t,x) = 0 \quad \text{ for each }c > {c}_{\hat{c},\hat\mu'}.
\end{equation}
  Letting $\hat\mu' \to \lambda_{\textup{LLW}}( \hat{c} - {c}_{\textup{LLW}})$, by direct calculation we have
  $$
   \lambda_{\hat c,\hat\mu'}=\frac{1}{2}\left(\hat c-\sqrt{\hat c^2-4(\hat\mu'+1-a)}\right)\to \lambda_{\textup{LLW}},
  $$
  so that
$${c}_{\hat{c},\hat\mu'}=  \lambda_{\hat c,\hat\mu'}+\frac{1-a}{ \lambda_{\hat c,\hat\mu'}}\to {c}_{\textup{LLW}}.$$
Hence, we deduce that \eqref{eq:theabove} holds for each $c > {c}_{\textup{LLW}}$.
The proof of Lemma \ref{lem:appen1} is complete.
\end{proof}

\section{Proof of Proposition \ref{prop:B2}} \label{sec:A}
This section is devoted to  the proof of Proposition \ref{prop:B2}.

Let $c_1 > \tilde{c}_1 \geq 2$ be given and let  $J_i(t,x)$ be given by \eqref{eq:B3-1'}, we may equivalently write
\begin{equation}\label{eq:jjj}
J_i(t,x) = \inf_{\gamma(\cdot)\in \mathbb{Y}^{t,x}}\left\{ \int_0^t L_i(s,\gamma(s),\dot\gamma(s))\,ds\right\},
\end{equation}
where $L_i$ is given in \eqref{eq:legendre}, and $\mathbb{Y}^{t,x}=\{\gamma \in H^1([0,t]): \, \gamma(0) \leq 0,~\gamma(t)=x\}.$

\begin {proof}[Proof of Proposition \ref{prop:B2}\rm{(a)}] We  divide the proof into several steps.

\noindent{\bf Step 1.}  We claim that for any $(t,x)\in(0,\infty)\times \mathbb{R}$, there exists some $\hat\gamma=\hat\gamma^{t,x}\in\mathbb{Y}^{t,x}$ such that
$$J_1(t,x) =  \int_0^t L_1(s,\hat\gamma(s),\dot{\hat\gamma}(s))\,ds.$$

Fix any $(t,x)\in(0,\infty)\times \mathbb{R}$. For each $k\geq 1$,  by \eqref{eq:jjj}, there is some $\gamma_k\in\mathbb{Y}^{t,x}$ such that
\begin{equation}\label{gamma_k}
  \int_0^t L_1(s,\gamma_k(s),\dot{\gamma}_k(s))\,ds\leq J_1(t,x)+1/k.
\end{equation}
 We claim that $\{\gamma_k\}_{k=1}^\infty$ is uniformly bounded in $H^1([0,t])$.  This is the case since (i) $\{\dot{\gamma}_k\}$ is uniformly bounded in $L^2$ by definition of $L^1$, and (ii) $\gamma_k(t) = x$. By passing to a subsequence, we may assume further that
 there is some $\hat\gamma\in\mathbb{Y}^{t,x}$ such that $\gamma_k\rightharpoonup\hat\gamma$ in $H^1([0,t]).$
Letting $k\to\infty$ in \eqref{gamma_k}, we  therefore arrive at
$$ J_1(t,x)\geq \liminf_{k\to\infty} \int_0^t L_1(s,\gamma_k(s),\dot{\gamma}_k(s))\,ds\geq  \int_0^t L_1(s,\hat\gamma(s),\dot{\hat\gamma}(s))\,ds\geq J_1(t,x),$$
where the last inequality follows from \eqref{eq:jjj}. Step 1 is thereby completed.

\noindent{\bf Step 2.} Let  $\hat\gamma\in \mathbb{Y}^{t,x}$ be given in Step 1. We show $\hat\gamma(0)=0$ if $x\geq0$.

Set $t_1=\inf\{s \in [0,t]: \hat\gamma(s)\geq 0\}$.
Define another path $\tilde{\gamma}\in \mathbb{X}$ by 
 $$
\tilde{\gamma}(s) = 0 ~\text{ for }s \in [0,t_1],\quad \tilde\gamma(s) = \hat\gamma(s) ~ \text{ for }s \in (t_1,t],
$$then
\begin{align*}
&\quad\int_0^t L_1(s,\hat\gamma(s),\dot{\hat\gamma}(s))\,ds\\
&=\int_0^{t_1} \left[\frac{|\dot {\hat{\gamma}}(s)|^2}{4} - 1+a \chi_{\{\hat{\gamma}(s) < c_1s\}}\right]ds + \int_{t_1}^t \left[\frac{|\dot {\hat{\gamma}}(s)|^2}{4}-1+a \chi_{\{\hat{\gamma}(s) < c_1s\}}\right]  ds\\
&\geq \int_0^{t_1} \left[ - 1+a\right]ds+\int_{t_1}^t \left[\frac{|\dot {\hat{\gamma}}(s)|^2}{4}-1 + a \chi_{\{\hat{\gamma}(s) < c_1s\}} \right]ds \\
&=\int_0^t L_1(s,\tilde\gamma(s),\dot{\tilde\gamma}(s))\,ds.
\end{align*}
Since $\hat\gamma$ is the minimizer, it follows that equality must hold, so that $\int_0^{t_1} |\dot{\hat\gamma}(s)|^2\,ds =0$, and thus $\hat\gamma(0) = \hat\gamma(t_1) =0$.

\noindent{\bf Step 3.} For $\frac{x}{t}\geq c_1$, we show
$
J_1(t,x)=J_2(t,x) =
\frac{t}{4}\left( \frac{x^2}{t^2} - 4\right).
$

We only show $J_1(t,x) =
\frac{t}{4}\left( \frac{x^2}{t^2} - 4\right)$, as the other one follows from the same arguments. By H\"{o}lder inequality, $J_1(t,x)\geq \frac{1}{4t} \left(\int_0^t |\dot\gamma(s)|\,ds\right)^2 - \int_0^t\,ds = \frac{x^2}{4t}-t$. 
Since the infimum can be attained by the path $\hat\gamma(s)=\frac{x}{t} \cdot s$ for $s\in[0,t]$, $J_1(t,x) =
\frac{t}{4}\left( \frac{x^2}{t^2} - 4\right)$ holds true.

\noindent{\bf Step 4.} For $0\leq \frac{x}{t}\leq c_1$, let $\hat{\gamma}$ be given in Step 1, and define
\begin{equation}\label{eq:enter}
\tau=\sup\left\{s \in [0,t]:\hat\gamma(s)\geq c_1s\right\}.
\end{equation} We show $\hat\gamma = \gamma_1$, where
$$
\gamma_1(s)=\left\{
\begin{array}{ll}
\medskip
c_1s,& 0\leq s\leq\tau,\\
c_1\tau+\frac{s-\tau}{t-\tau}(x-c_1\tau), &\tau<s\leq t.
\end{array}
\right.
$$

Since $\hat\gamma(\tau) = c_1\tau$, we have
\begin{equation}\label{eq:cs1}
\int_{0}^{\tau} \frac{|c_1|^2}{4}\,ds = \frac{1}{4\tau}\left[ \int_{0}^{\tau} \dot{\hat\gamma}(s)\,ds\right]^2 \leq \int_{0}^{\tau} \frac{|\dot{\hat\gamma}(s)|^2}{4}\,ds,
\end{equation}
and
\begin{equation}\label{eq:cs2}
\int_{\tau}^t \frac{1}{4}\left|\frac{x-c_1\tau}{t-\tau} \right|^2\,ds = \frac{1}{4(t-\tau)}\left[ \int_{\tau}^t \dot{\hat\gamma}(s)\,ds\right]^2 \leq \int_{\tau}^{t} \frac{|\dot{\hat\gamma}(s)|^2}{4}\,ds.
\end{equation}
Suppose $\hat\gamma \not\equiv \gamma_1$, then one of \eqref{eq:cs1} and \eqref{eq:cs2} is the strict inequality, so that
\begin{align*}
&\quad\int_0^t L_1(s, \gamma_1(s),\dot{\gamma}_1(s))\,ds\\
&=\int_0^\tau \left[\frac{|c_1|^2}{4} - 1\right]ds +\int_\tau^t \left[\frac{1}{4}\left| \frac{x-c_1\tau}{t-\tau}\right|^2 - 1 + a\right]\,ds \\
&< \int_0^\tau \left[\frac{|\dot{\hat\gamma}(s)|^2}{4} - 1\right]ds +\int_\tau^t \left[\frac{| \dot{\hat\gamma}(s)|^2}{4} - 1 + a\right]\,ds \\
&=\int_0^t L_1(s, \hat\gamma(s),\dot{\hat\gamma}(s))\,ds.
\end{align*}
This is a contradiction to definition of $\hat\gamma$, so that $\hat\gamma\equiv \gamma_1$.

\noindent{\bf Step 5.}  For $\frac{x}{t}\leq c_1$, let  $\hat\gamma$ be given in Step 1. 
We show $\hat\gamma(s)\leq c_1 s$ for $s \in [0,t]$.

We consider respectively two cases: \rm{(i)} $0\leq \frac{x}{t}\leq c_1$ and \rm{(ii)} $\frac{x}{t}<0$. For \rm{(i)}, by Step 4, we can directly get $\hat\gamma(s)\leq c_1 s$ for $s \in [0,t]$ by the explicit minimizing path determined there. For  \rm{(ii)}, if $\hat\gamma(s)\leq 0$  for $s \in [0,t]$, then there is nothing to prove; Otherwise,
there exists some $\tau'\in [0,t)$ such that $\hat\gamma(\tau')=0$ and $\hat\gamma(s)<0$ for $s\in (\tau',t]$. By the dynamic programming principle, we rewrite $J_1$ as
  \begin{align*}
J_1(t,x)&=\int_{0}^{\tau'}L_1(s, \hat\gamma(s),\dot{\hat\gamma}(s))\,ds+\int_{\tau'}^t L_1(s, \hat\gamma(s),\dot{\hat\gamma}(s))\,ds\\
&= \inf_{\gamma(\cdot)\in \mathbb{Y}^{\tau',0}}\left\{\int_{0}^{\tau'} L_1(s,\gamma(s),\dot\gamma(s))\,ds\right\}+\int_{\tau'}^t L_1(s, \hat\gamma(s),\dot{\hat\gamma}(s))\,ds.
\end{align*}
Then by Step 4, we deduce $\hat\gamma(s)\leq c_1 s$ for $s \in [0,\tau']$. This together with definition of $\tau'$, implies $\hat\gamma(s)\leq c_1s$ for $s\in [0,t]$, which completes Step 5.

\noindent{\bf Step 6.} For $x<0$, we show $J_1(t,x)=-t(1-a)$.

It follows from Step 5 that  the minimizing path $\hat{\gamma}$ stays in $\left\{(t,x):\, x\leq c_1t\right\}$. Hence $J_1(t,x) \geq \int_0^t(-1 + a)\,ds = -t(1-a)$. On the other hand, the infimum  is attained by the constant path $\hat \gamma(s)\equiv x$ for $s\in[0,t]$. Therefore,  $J_1(t,x)=-t(1-a)$.

\noindent{\bf Step 7.} We verify Proposition \ref{prop:B2}\rm{(a)}, i.e., \eqref{eq:propA1}.

 By Step 3 and Step 6, it remains to consider the case $0\leq\frac{x}{t}<c_1$. In this case,  if $c_1-2\sqrt{a}\leq \frac{x}{t}<c_1$, by Step 4, we have $\hat \gamma=\gamma_1$ and thus \begin{equation}\label{eq:gamma1}
\begin{split}
&\quad J_1(t,x) \\
&=  \inf\limits_{0 \leq \tau < t} \left\{\frac{(x-c_1\tau)^2}{4(t-\tau)} - (1-a)(t-\tau) + \tau\left( \frac{c_1^2}{4}  - 1\right)\right\}\\
&=\inf\limits_{0 < s \leq t} \left\{\frac{(x-c_1t)^2}{4s}+\frac{c_1(x-c_1t)}{2} - t+as +  \frac{c_1^2}{4}t\right\}\\
&=\left[\frac{c_1}{2}-\sqrt{a}\right][x-\bar{c}_{\textup{nlp}}t],
\end{split}
\end{equation}
where $\bar{c}_{\textup{nlp}}=\frac{c_1}{2}-\sqrt{a}+\frac{1-a}{\frac{c_1}{2}-\sqrt{a}}$.

On the other hand,  if $0\leq\frac{x}{t}<c_1-2\sqrt{a}$, then from the calculation above, $\frac{(x-c_1t)^2}{4(t-\tau)}+a(t-\tau)$ is an increasing function of $\tau$ when $0\leq\frac{x}{t}<c_1-2\sqrt{a}$.
So the infimum is attained at $\tau=0$, whence by the first equality of \eqref{eq:gamma1}, we directly obtain
\begin{equation*}
J_1(t,x) = \frac{t}{4}\left[\frac{x^2}{t^2}-4(1-a)\right]. 
\end{equation*}

The proof of Proposition \ref{prop:B2}{\rm{(a)}} is now complete.
\end{proof}

\begin {proof}[Proof of Proposition \ref{prop:B2}\rm{(b)}]
The Friedlin's condition \eqref{eq:freidlin} is a direct consequence of the following two observations:
\begin{itemize}

\item[{\rm (i)}] (by \eqref{eq:propA1}) There exists $c_0  \in (0, c_1)$ such that $$P=\{(t',x'): J_1(t',x') >0,\, t'>0\} = \{(t',x'): x' > c_0t',\, t'>0\}.$$
\item[{\rm (ii)}]  Since all possibilities are considered in the proof of Proposition \ref{prop:B2}\rm{(a)}, we can conclude that for each $(t,x) \in (0,\infty)\times (0,\infty)$ the optimal path $\hat\gamma=\hat\gamma^{t,x}$ of $J_1(t,x)=0$ is a piecewise line curve connecting $(0,0)$,   $(\tau,c_1\tau)$ and $(t,x)$ for some  $\tau \in [0,t)$. In particular the Freidlin condition \eqref{eq:freidlin} holds for $(t,x) \in \partial P= \{(t',x'):x'= c_0t' \}$.
\end{itemize}
The proof is now complete.
\end{proof}

\begin {proof}[Proof of Proposition \ref{prop:B2}\rm{(c)}] Let $c_1>\tilde{c}_1\geq 2$ be given.
First, observe from definition of $J_i$ in \eqref{eq:jjj} that
\begin{equation}\label{eq:jj}
J_i(kt,kx)=kJ_i(t,x) \quad \text{ for } k>0,\,t>0,\, x\in \mathbb{R},\,\text{ and }\, i=1,2.
\end{equation}

Next, we claim that there exists some  $\delta^*>0$ such that $J_1(t,x)=J_2(t,x)$ for all  $ (t,x)\in B_{\delta^*}(1,c_1)$, where $B_{\delta^*} (1,c_1)$ is a disk in $\mathbb{R}^2$ with center $(1,c_1)$ and radius $\delta^*$.

Fix $(t,x)\in (0,\infty)\times \mathbb{R}$ and let $\hat{\gamma}=\hat{\gamma}^{t,x}$ be the minimizing path of $J_1(t,x)$ for $ (t,x)\in B_{\delta^*}(1,c_1)$. We claim that  $\hat{\gamma}$ is also the minimizing path of $J_2(t,x)$.  To do so, define 
$$
\mathbb{Y}^{t,x}_1=\left\{\gamma\in H^1([0,t])\,\Big|\,\gamma(0)\leq 0, \gamma(t) = x, \text{ and }\gamma(s)\leq \tilde c_1s \text{ for some }s\right\}.
$$
Let $B_{\delta^*}^+(1,c_1):=\{(t,x)\in B_{\delta^*}(1,c_1)\,|\,\frac{x}{t}\geq c_1\}.$
By Step 2 in the proof of Proposition \ref{prop:B2}\rm{(a)}, 
we have, for all $\delta^*>0$, that
$$
J_1(t,x)=J_2(t,x) \quad \text{ in }B_{\delta^*}^+(1,c_1).
$$
Also notice from Step 4 in the proof of Proposition \ref{prop:B2}\rm{(a)}, 
that $\hat{\gamma}^{1,c_1}(s)=c_1s$, so that $\hat{\gamma}^{1,c_1}\notin\mathbb{Y}^{1,c_1}_1$ (since $c_1 > \tilde{c}_1$) and thus there exists some $\delta_0>0$ such that
\begin{equation*}
J_2(1,c_1)=\frac{1}{4}(c_1^2-4)< \inf\limits_{\gamma \in \mathbb{Y}^{1,c_1}_1}\left\{ \int_0^1 \left[\frac{|\dot \gamma(s)|^2}{4} - 1\right]ds \right\}-\delta_0.
\end{equation*}
By the continuity of $J_2$, we choose $\delta^*>0$ so that for $(t,x)\in B_{\delta^*}(1,c_1) \setminus B_{\delta^*} ^+(1,c_1)$,
\begin{align*}
J_2(t,x)&\leq J_2(1,c_1)+\frac{\delta_0}{2}\\
&\leq\inf\limits_{\gamma \in \mathbb{Y}^{1,c_1}_1} \int_0^t \left[\frac{|\dot \gamma(s)|^2}{4} - 1\right]ds-\frac{\delta_0}{2}\\
&<\inf\limits_{\gamma \in \mathbb{Y}^{t,x}_1} \int_0^t \left[\frac{|\dot \gamma(s)|^2}{4} - 1\right]ds\\
&\leq \inf\limits_{\gamma \in  \mathbb{Y}^{t,x}_1} \int_0^t L_2(s,\gamma(s),\dot\gamma(s))\,ds,
\end{align*}
which implies $\hat{\gamma}^{t,x}\in\mathbb{Y}^{t,x}\backslash \mathbb{Y}^{t,x}_1$, i.e,  the minimizing path stays in $\left\{x > \tilde c_1t\right\}$ and hence $J_1=J_2$ in $B_{\delta^*}(1,c_1)$.

Taking \eqref{eq:jj} into account, we conclude that for $(t,x)\in B_{\delta^*}(1,c_1)$ and $k>0$,
$$J_1(kt,kx)=kJ_1(t,x)=kJ_2(t,x)=J_2(kt,kx),$$
which implies immediately that $J_1(t,x)=J_2(t,x)$ in $\left\{x\geq (c_1-\delta^*)t\right\}$. 
\end{proof}

\end{appendices}



\medskip
Received xxxx 20xx; revised xxxx 20xx.
\medskip

\end{document}